\documentclass{article}

\usepackage{cite}
\usepackage{amssymb}
\usepackage{amsmath}
\usepackage{amsfonts}
\usepackage{amsthm}
\usepackage{enumerate}
\usepackage{colonequals}
\usepackage{mathrsfs} 
\usepackage[dvipsnames]{xcolor}
\usepackage{enumitem}
\usepackage{mathtools}
\usepackage[all,cmtip]{xy}
\usepackage{tikz-cd} 
\usepackage{multicol}
\usepackage{float}
\usepackage{placeins}
\usepackage{accents}

\usetikzlibrary{decorations.markings}

\usepackage{url}

\usepackage{euscript}
\usepackage{wasysym}

\newtheorem{theorem}{Theorem}[section]

\newtheorem{proposition}{Proposition}[section]

\newtheorem*{theorem*}{Theorem}
\newtheorem{corollary}{Corollary}[section]
\newtheorem{definition}{Definition}[section]

\newtheorem{remark}{Remark}[section]
\newtheorem{conjecture}{Conjecture}[section]

\newcommand{\eusc}[1]{\EuScript{#1}}

\numberwithin{equation}{section}

\renewcommand{\sf}[1]{\mathsf{#1}}

\allowdisplaybreaks[2]

\newcommand{\cara}{Carath\'{e}odory\,\,}

\newcommand{\deee}{\hspace{2 pt} \mathrm{d}}

\newcommand{\ov}[1]{\overline{#1}}

\newcommand{\mrm}[1]{\mathrm{#1}}
\newcommand{\nml}{\left \vert \left \vert}
\newcommand{\nmr}{\right \vert \right \vert}

\renewcommand{\hat}{\widehat}

\newcommand{\fgroup}{\mathfrak{F}_{2,3}}
\newcommand{\Cstar}{C^*}

\DeclareMathOperator{\re}{\mathrm{Re}}

\DeclareMathOperator{\prob}{\mathrm{Prob}}

\DeclareMathOperator{\bbt}{\mathbb{T}}
\DeclareMathOperator{\bbz}{\mathbb{Z}}

  \usepackage{stackengine,scalerel}

\allowdisplaybreaks

\begin{document}

\title{Formulations of Furstenberg's $\times 2 \times 3$ conjecture in complex analysis and operator algebras}
\author{Peter Burton and Jane Panangaden}

\maketitle

\begin{abstract} Furstenberg's $\times 2 \times 3$ conjecture has remained a central open problem in ergodic theory for over $50$ years, and it serves as the basic test case for a broad class of rigidity phenomena which are believed to hold in number-theoretic dynamics. More recently, two related statements have appeared in the literature: a question about periodic approximation raised by Levit and Vigdorovich in the context of approximate group theory and a periodic equidistribution conjecture formulated by Lindenstrauss. The purpose of this article is to provide equivalent formulations for these three statements in a complex-analytic setting and an operator-algebraic setting, giving nine conjectures grouped into three triples. The complex-analytic conjectures involve so-called \cara functions on the unit disk that satisfy a certain functional identity, and we find that Furstenberg's conjecture is equivalent to the assertion that every such function is a convex combination of rational functions. The operator-algebraic conjectures involve tracial states on the full group $C^\ast$-algebra of a certain semidirect product, which is related to Baumslag-Solitar groups.

\end{abstract}

\tableofcontents

\section{Introduction}

\subsection{Nine conjectures}

\subsubsection{Ergodic theory on the circle} \label{sec.ergintro}

In Section \ref{sec.ergintro} we state the three preexisting ergodic-theoretic conjectures which motivate our work. Write $\mathbb{T}$ for the unit circle in the complex plane and write $\prob(\mathbb{T})$ for the set of all Borel probability measures on $\mathbb{T}$. We now stipulate some definitions regarding a notion of invariance for measures on $\mathbb{T}$ with respect to taking powers of the variable.
\begin{definition} \label{invdef}
For $n \in \mathbb{N}$, we define a measure $\mu \in \prob(\mathbb{T})$ to be $\times n$\emph{-invariant} if the following holds for all continuous functions $f:\mathbb{T} \to \mathbb{C}$.
\begin{equation*} \int_\mathbb{T} f(\omega^n) \deee \mu(\omega) = \int_\mathbb{T} f(\omega) \deee \mu(\omega)  \end{equation*}	
Now, let $n_1,\ldots,n_m \in \mathbb{N}$. \begin{itemize} \item  We define a measure $\mu \in \prob(\mathbb{T})$ to be $\times(n_1,\ldots,n_m)$\emph{-invariant} if $\mu$ is $\times n_k$-invariant for all $k \in \{1,\ldots,m\}$, and we write $\prob(\mathbb{T} ; n_1,\ldots,n_m)$ for the set of all $\times(n_1,\ldots,n_m)$-invariant elements of $\prob(\mathbb{T})$. 
\item A measure $\mu \in \prob(\mathbb{T} ; n_1,\ldots,n_m)$ is defined to be \emph{ergodic} if every expression $\mu = t \kappa + (1-t)\nu$ for $\kappa,\nu \in \prob(\mathbb{T} ; n_1,\ldots,n_m)$ and $t \in [0,1]$ is trivial in the sense that $t \in \{0,1\}$. \end{itemize} \end{definition}

For any element $\mu$ of $\prob(\mathbb{T})$ there exists a unique minimal closed subset $E$ of $\mathbb{T}$ such that $\mu(E)=1$. As usual, we refer to this set as the \emph{support of} $\mu$, and denote it by $\operatorname{supp}(\mu)$. We also write $\lambda$ for the Lebesgue probability measure on $\mathbb{T}$. The following conjecture appeared for the first time in \cite{MR213508}, and has remained central in ergodic theory since then. It also appears as Conjecture 3 of \cite{MR2827904}, where Lindenstrauss provides extensive context situating it in homogenous dynamics and number theory.  

\begin{conjecture}[Furstenberg's $\times (2,3)$ conjecture] \label{conj.erg-1} \label{conj.furst} Let $\mu \in \prob(\mathbb{T};2,3)$ be ergodic. Then either $\mu = \lambda$ or $\operatorname{supp}(\mu)$ is a finite set. \end{conjecture} 

\begin{remark} \label{rem.six} Conjecture \ref{conj.furst} has a version for $\times (a,b)$-invariant measures whenever $a,b \in \mathbb{N}$ are multiplicatively independent. For clarity of exposition, in this article we focus on the $\times (2,3)$ case. However, all our results can be extended to the general case by directly substituting $a$ and $b$ for $2$ and $3$ respectively. \end{remark}

 We recall that the \emph{vague topology} on $\prob(\mathbb{T})$ is defined by the stipulation that a sequence $(\mu_n)_{n \in \mathbb{N}}$ of elements of $\prob(\mathbb{T})$ converges vaguely to $\mu \in \prob(\mathbb{T})$ if and only if the following holds for all continuous functions $f:\mathbb{T} \to \mathbb{C}$.
\begin{equation*} \lim_{n \to \infty} \int_\mathbb{T} f(\omega) \deee \mu_n(\omega) = \int_\mathbb{T} f(\omega) \deee \mu(\omega) \end{equation*}

The following is posed as Question 4 in \cite{MR4753084}, where it is connected to the notion of Hilbert-Schmidt stability for sequence of approximate representations. We formulate it as a conjecture.

\begin{conjecture}[Periodic approximation conjecture, after Levit and Vigdorovich] \label{conj.erg-2} Every $\times(2,3)$-invariant measure on $\mathbb{T}$ is a vague limit of finitely supported $\times(2,3)$-invariant measures. \end{conjecture}

The following statement was formulated as a question by Bourgain in \cite{MR2534239} and appears as Conjecture 6 in \cite{MR2827904}.

\begin{conjecture}[Bourgain-Lindenstrauss periodic equidistribution conjecture] \label{conj.erg-3} Let $(\mu_n)_{n \in \mathbb{N}}$ be a sequence of ergodic finitely supported elements of $\prob(\mathbb{T};2,3)$ and assume that $\lim_{n \to \infty} \lvert \operatorname{supp}(\mu_n)\rvert = \infty$. Then $(\mu_n)_{n \in \mathbb{N}}$ converges in the vague topology  to $\lambda$ as $n \to \infty$. \end{conjecture}

The above statement of Conjecture \ref{conj.erg-3} has a superficial difference from the statement in \cite{MR2827904}. It is straightforward to see that this difference is irrelevant, and we prove details in Section \ref{sec.ooska} below. 

\subsubsection{Carath\'{e}odory functions} \label{sec.carinto}

In Section \ref{sec.carinto} we formulate three complex-analytic conjectures which will be shown to be equivalent to the ergodic-theoretic conjectures from Section \ref{sec.ergintro}. Write $\mathbb{D}$ for the unit disk in the complex plane. A \emph{Carath\'{e}odory function} is defined to be a holomorphic function $\psi:\mathbb{D} \to \mathbb{C}$ such that $\re(\psi(z)) >0$ for all $z \in \mathbb{D}$ and such that $\psi(0) = 1$. We write $\operatorname{Cara}(\mathbb{D})$ for the set of all \cara functions. We now stipulate some definitions regarding a condition on Carath\'{e}odory functions that relates taking powers of the variable to averages over rotates.
\begin{definition}
For $n \in \mathbb{N}_+$, we define a function $\psi:\mathbb{D} \to \mathbb{C}$ to be $\times n$-\emph{circular} if the following holds for all $z \in \mathbb{D}$. \[ \psi(z^n) = \frac{1}{n} \sum_{k=0}^{n-1} \psi(e^{2 \pi i k/n} z) \]
Now, let $n_1,\ldots,n_m \in \mathbb{N}$. \begin{itemize} \item  We define a function $\psi:\mathbb{D} \to \mathbb{C}$ to be $\times(n_1,\ldots,n_m)$\emph{-circular} if $\psi$ is $\times n_k$-circular for all $k \in \{1,\ldots,m\}$. We also write $\operatorname{Cara}(\mathbb{D} ; n_1,\ldots,n_m)$ for the set of all $\times(n_1,\ldots,n_m)$-circular elements of $\operatorname{Cara}(\mathbb{D})$. 
\item A function $\psi \in \operatorname{Cara}(\mathbb{D} ; n_1,\ldots,n_m)$ is defined to be \emph{extreme} if every expression $\psi = t \alpha + (1-t)\beta$ for $\alpha,\beta \in \operatorname{Cara}(\mathbb{D} ; n_1,\ldots,n_m)$ and $t \in [0,1]$ is trivial in the sense that $t \in \{0,1\}$. \end{itemize} \end{definition}

We now formulate our three complex-analytic conjectures.

\begin{conjecture}\label{conj.comp-1} Every extreme element of $\operatorname{Cara}(\mathbb{D} ; 2,3)$ is rational. \end{conjecture}

\begin{conjecture} \label{conj.comp-2}  Every element of $\operatorname{Cara}(\mathbb{D} ; 2,3)$ is a compact-uniform limit of rational elements of \linebreak $\operatorname{Cara}(\mathbb{D} ; 2,3)$. \end{conjecture}

\begin{conjecture} \label{conj.comp-3} Suppose that $(\psi_n)_{n \in \mathbb{N}}$ is a sequence of rational extreme elements of $\operatorname{Cara}(\mathbb{D} ; 2,3)$, and assume that number of poles of $\psi_n$ diverges to infinity as $n \to \infty$. Then $(\psi_n)_{n \in \mathbb{N}}$ converges in the compact-uniform topology to the constant function with value $1$ as $n \to \infty$. \end{conjecture}

\subsubsection{Tracial states on a number-theoretic $C^\ast$-algebra} \label{sec.traceintro}

In Section \ref{sec.traceintro} we formulate three operator-algebraic conjectures which will again be shown to be equivalent to the ergodic-theoretic conjectures from Section \ref{sec.ergintro}. Given a $C^\ast$-algebra $S$, we write $\operatorname{St}(S)$ for the set of all states on $S$, and $\operatorname{TSt}(S)$ for the set of all tracial states on $S$. We also define a sequence of states $(\Phi_n)_{n \in \mathbb{N}}$ on $S$ to \emph{converge pointwise} to a state $\Phi$ if we have $\lim_{n \to \infty} \Phi_n(s) = \Phi(s)$ for all $s \in S$, and we define an element $\Phi$ of $\operatorname{TSt}(S)$ to be \emph{extreme} if every expression $\Phi = t \Theta + (1-t)\Omega$ for $\Theta,\Omega \in \operatorname{TSt}(S)$ and $t \in [0,1]$ is trivial in the sense that $t \in \{0,1\}$. If $G$ is a countable discrete group, we adopt the notation $C^\ast(G)$ for the full group $C^\ast$-algebra of $G$, and we write $\upsilon_G:G \hookrightarrow C^\ast(G)$ for the canonical embedding of $G$ into the unitary group of $C^\ast(G)$. \\
\\
Let $n \in \mathbb{N}$ and suppose $a \in \mathbb{N}$ is a divisor of $n$. We define the $\times a$\emph{-automorphism of }$\mathbb{Z}[1/n]$ to be the map $\sf{T}_a: \mathbb{Z}[1/n] \to \mathbb{Z}[1/n]$ given by $\sf{T}_a(q) = aq$ for $q \in \mathbb{Z}[1/n]$. It is immediate that for any two divisors $a$ and $b$ of $n$, the automorphisms $\sf{T}_a$ and $\sf{T}_b$ commute. Thus we can define the \emph{Furstenberg group of type }$(2,3)$ to be the semidirect product $\mathfrak{F}_{2,3} = \mathbb{Z}^2 \ltimes \mathbb{Z}[1/6]$, where the two generators of $\mathbb{Z}^2$ acts on $\mathbb{Z}[1/6]$ according to $\sf{T}_2$ and $\sf{T}_3$ respectively. Thus the operation in $\mathfrak{F}_{2,3}$ is given explicitly for $j,k,\ell,m \in \mathbb{Z}$ and $p,q \in \mathbb{Z}[1/6]$ by: 
\begin{align*} (j,k;p) \star (\ell,m;q) \coloneqq (j+\ell, k+m; p + 2^j3^kq) \end{align*}
If $\mu$ were assumed to be $\times n$-invariant for a single $n$, the analogous group $\mathbb{Z} \ltimes \mathbb{Z}[1/n]$ would reduce to a Baumslag-Solitar group. The $C^\ast$-algebra $C^\ast(\mathfrak{F}_{2,3})$ comes equipped with the \emph{regular tracial state}, given by linear extension of the map $\Delta:\mathfrak{F}_{2,3} \to \{0,1\}$ defined for $j,k \in \mathbb{Z}$ and $q \in \mathbb{Z}[1/6]$ as follows. 
\begin{equation*} \Delta(j,k,q) = \begin{cases} 1 & \mbox{ if } j = k = q = 0 \\ 0 & \mbox{ else} \end{cases} \end{equation*}
We will continue to denote the regular tracial state on $C^\ast(\mathfrak{F}_{2,3})$ by $\Delta$. We will also need the following concept.

\begin{definition} Regarding $C^\ast(\mathbb{Z}[1/6])$  as a subalgebra of $C^\ast(\mathfrak{F}_{2,3})$, we define the \emph{conditional dimension} of a state $\Phi \in \operatorname{St}(\mathfrak{F}_{2,3})$ to be the dimension of the image of $C^\ast(\mathbb{Z}[1/6])$ in the GNS representation of $C^\ast(\mathfrak{F}_{2,3})$ associated with $\Phi$. If $\Phi$ is conditonally finite-dimensional, we write $\operatorname{cndm}(\Phi)$ for the natural number corresponding to the conditional dimension of $\Phi$.  \end{definition}

\begin{remark} Theorem 4.1 in \cite{MR626496} implies that there exists a conditional expectation $\mathbb{E}_{2,3}: C^\ast(\mathfrak{F}_{2,3}) \twoheadrightarrow C^\ast(\mathbb{Z}[1/6])$. Thus the dimension of a state $\Phi \in \operatorname{St}(C^\ast(\mathfrak{F}_{2,3}))$ may be alternatively defined as the dimension of the range of $\mathbb{E}_{2,3}$ in the GNS representation of $C^\ast(\mathfrak{F}_{2,3})$ associated with $\Phi$. \end{remark}

We now formulate our three operator-algebraic conjectures.

\begin{conjecture} \label{conj.see-1} Let $\Phi \in \operatorname{TSt}(C^\ast(\mathfrak{F}_{2,3}))$ be extreme. Then either $\Phi = \Delta$ or $\Phi$ is conditionally finite-dimensional. \end{conjecture} 

\begin{conjecture} \label{conj.see-2}  Every element of $\operatorname{TSt}(C^\ast(\mathfrak{F}_{2,3})$ is a pointwise limit of conditionally finite-dimensional elements of $\operatorname{TSt}(C^\ast(\mathfrak{F}_{2,3}))$. \end{conjecture}

\begin{conjecture} \label{conj.see-3}  Suppose that $(\Phi_n)_{n \in \mathbb{N}}$ is a sequence of extreme conditionally finite-dimensional elements of $\operatorname{TSt}(C^\ast(\mathfrak{F}_{2,3}))$ and assume that $\lim_{n \to \infty}\operatorname{cndm}(\Phi_n) = \infty$. Then $(\Phi_n)_{n \in \mathbb{N}}$ converges pointwise to $\Delta$ as $n \to \infty$. \end{conjecture}

\subsection{Statement of results}

The purpose of this article is to prove the following results.

\begin{theorem} \label{thm.1} Conjectures \ref{conj.erg-1}, \ref{conj.comp-1} and \ref{conj.see-1} are equivalent.
	
\end{theorem}

\begin{theorem} \label{thm.2} Conjectures \ref{conj.erg-2}, \ref{conj.comp-2} and \ref{conj.see-2} are equivalent.
	
\end{theorem}

\begin{theorem} \label{thm.3} Conjectures \ref{conj.erg-3}, \ref{conj.comp-3} and \ref{conj.see-3} are equivalent. \end{theorem}

Corollary 3.7 of \cite{MR4706815} shows that Conjecture \ref{conj.erg-1} is equivalent to the absence of faithful tracial states on $C^\ast(\mathfrak{F}_{2,3})$ other than $\Delta$. This result is very close to the equivalence between Conjectures \ref{conj.erg-1} and \ref{conj.see-1}. A related theorem is the main result of \cite{MR3679610}. However, these references do not attempt to reformulate Conjectures \ref{conj.erg-2} or \ref{conj.erg-3} in the operator-algebraic setting.\\
\\
The equivalence of these conjectures will be shown in two parts: the ergodic and complex-analytic portions in Section \ref{sec.meas-fun} and the ergodic and operator-algebraic portions in Section \ref{sec.tracial}.

\section{Measure-function correspondence} \label{sec.meas-fun}

\subsection{Introducing the Herglotz correspondence}

In Section \ref{sec.meas-fun} we establish the existence of a certain affine homeomorphism between the space of Borel probability measures on the unit circle with the vague topology and the set of \cara functions with the compact-uniform topology, which will be referred to as the Herglotz correspondence. The Herglotz correspondence itself is classical, but the fact that it is a homeomorphism between these topologies appears to be novel, so we take care in expositing details. This correspondence is the main idea  behind the equivalence of Conjectures \ref{conj.erg-1} and \ref{conj.comp-1}, the equivalence of Conjectures \ref{conj.erg-2} and \ref{conj.comp-2} and the equivalence of Conjectures \ref{conj.erg-3} and \ref{conj.comp-3}.  

\subsubsection{Measures to functions}

We introduce the \emph{Herglotz kernel on the unit disk}, which is the function $\mathcal{H}: \mathbb{T} \times \mathbb{D} \to \mathbb{C}$ given as follows.
\begin{align*} \mathcal{H}(e^{i\theta},z) \coloneqq \frac{e^{i\theta}+z}{e^{i\theta}-z} \end{align*}
For $\mu \in \prob(\mathbb{T})$, we define a function $\psi^\mu:\mathbb{D} \to \mathbb{C}$ as follows.
\begin{align} \label{eq.psimu-def} \psi^{\mu}(z) \coloneqq \int_{\bbt} \mathcal{H}(e^{i\theta}, z) \deee \mu(e^{i\theta}) \end{align}
We now verify that $\psi^\mu$ is indeed a \cara function.
\begin{description} \item[(Holomorphicity)] Since $\mathcal{H}(e^{i\theta},z)$ is a holomorphic function of $z$ for a fixed $e^{i\theta} \in \mathbb{T}$, is clear that $\psi^\mu$ is holomorphic. 
\item[(Positivity of real part)] There is a standard expression for the Poisson kernel on the unit disk as the real part of the Herglotz kernel. (See, for example, (1.3.14) in \cite{MR2105088}.) This expression asserts that the following holds for all $e^{i\theta},e^{i\varphi} \in \mathbb{T}$ and all $r \in [0,1)$.
\begin{align*} \re(\mathcal{H}(e^{i\theta},re^{i\varphi})) = \frac{1-r^2}{1-2\cos(\theta-\varphi)+r^2}  \end{align*}
It is transparent that the right side of the previous display is strictly positive and so we see that $\re(\psi^\mu(z)) > 0$ for all $z \in \mathbb{T}$. 
\item[(Normalization at zero)] Since $\mu$ is a probability measure we find:
\[ \psi_{\mu}(0) = \int_{\bbt} \mathcal{H}(e^{i\theta}, 0) \deee \mu(e^{i\theta})=  \int_{\bbt}  \deee \mu(e^{i\theta}) = 1 \]
\end{description}
Thus we have verified that $\psi_{\mu}$ is a \cara function, as required.

\subsubsection{Functions to measures}

Conversely, let $\psi: \mathbb{D} \rightarrow \mathbb{C}$ be a \cara function. For $r \in (0,1)$ we define a measure $\mu_{\psi, r}$ on $\mathbb{T}$ as follows.
\begin{align} \label{eq.mupsi-def} \deee \mu_{\psi, r}(e^{i\theta}) \coloneqq  \re(\psi(re^{i\theta})) \deee \lambda(e^{i \theta}) \end{align}
We now explain why $\mu_{\psi,r}$ is a probability measure. According to the Poisson representation, if $U$ is an open subset of $\mathbb{C}$ with $\ov{\mathbb{D}} \subseteq U$ and $g:U \to \mathbb{C}$ is holomorphic with $g(0) \in \mathbb{R}$ then the following holds for all $z \in \mathbb{D}$.
\begin{align} \label{eq.poisson-rep} g(z) =  \int_{\bbt}\frac{e^{i\theta}+z}{e^{i\theta}-z}\re(g(e^{i\theta})) \deee \lambda(e^{i \theta})\end{align}
This formula appears, for example, as (1.3.17) in \cite{MR2105088}. Now, let $r \in (0,1)$ and instantiate the previous display with the choice $g(z)\coloneqq \psi(rz)$. We obtain 
\[ 1 = \psi(0)  = \int_{\bbt} \re(\psi(re^{i\theta})) \deee \lambda(e^{i\theta})  = \int_{\bbt} \deee \mu_{\psi,r}(e^{i\theta}) \]
and so we find that $\mu_{\psi,r} \in \prob(\mathbb{T})$, as required. In Section 1.3.4 of \cite{MR2105088}, it is shown that the family of measures $(\mu_{\psi, r})_{r \in (0,1)}$ is convergent in the vague topology on $\prob(\mathbb{T})$ as $r\uparrow 1$. We denote the limiting measure by $\mu_{\psi}$.

\subsection{Exchange of Fourier coefficients and Taylor coefficients}

For $\ell \in \mathbb{Z}$, we adopt the convention below for the $\ell^{\mrm{th}}$ Fourier coefficient of a Borel measure $\mu$ on $\mathbb{T}$.
\begin{align} \hat{\mu}(\ell) \coloneqq \int_{\bbt} e^{-i\ell\theta} \deee \mu(e^{i \theta}) \label{eq.fou-con} \end{align}

\begin{remark} \label{rem.vague} The vague topology may be alternatively defined by the stipulation that a sequence $(\mu_n)_{n \in \mathbb{N}}$ of elements of $\prob(\mathbb{T})$ converges vaguely to $\mu \in \prob(\mathbb{T})$ if and only if $\lim_{n \to \infty} \hat{\mu_n}(\ell) = \hat{\mu}(\ell)$ for all $\ell \in \mathbb{Z}$. \end{remark}

Conversely, for $\ell \in \mathbb{N}$ and a holomorphic function $f: \mathbb{D} \to \mathbb{C}$ we let $a_{f}(\ell)$ denote $\ell^{\mrm{th}}$ Taylor coefficient of $f$.

\begin{proposition}[Fourier-Taylor correspondence] \label{prop.fou-tay} 
\begin{description}
\item[(Clause I)] For all $\mu \in \prob(\mathbb{T})$ and all $\ell \in \mathbb{Z}$ we have:
\begin{align*} \hat{\mu}(\ell)  = \begin{cases} \frac{1}{2} a_{\psi^\mu}(\ell) & \mbox{ if }\ell > 0
\\ 1 & \mbox{ if } \ell = 0
\\ \frac{1}{2} \ov{a_{\psi^\mu}(\ell)} & \mbox{ if } \ell < 0 \end{cases} \end{align*}
 
\item[(Clause II)] For all $\psi \in \operatorname{Cara}(\mathbb{D})$ and all $\ell \in \mathbb{Z}$ we have:
\begin{align*}  \hat{\mu_\psi}(\ell) =  \begin{cases} \frac{1}{2} a_{\psi}(\ell) & \mbox{ if }\ell > 0
\\ 1 & \mbox{ if } \ell = 0
\\ \frac{1}{2} \ov{a_{\psi}(\ell)} & \mbox{ if } \ell < 0 \end{cases}\end{align*} \end{description} \end{proposition}

We note the following corollary of Proposition \ref{prop.fou-tay}.

\begin{corollary} \label{cor.bij} The mappings $[\mu \mapsto \psi^\mu]: \prob(\mathbb{T}) \to \operatorname{Cara}(\mathbb{D})$ and $[\psi \mapsto \mu_\psi]: \operatorname{Cara}(\mathbb{D}) \to \prob(\mathbb{T})$ are mutual inverses. In particular, both of these mappings are bijections. \end{corollary}

We now turn to the proof of Proposition \ref{prop.fou-tay}.

\begin{proof}[Proof of Proposition \ref{prop.fou-tay}] We begin by observing that the following holds for all holomorphic functions $f:\mathbb{D} \to \mathbb{C}$, all $r \in (0,1)$ and all $e^{i \theta} \in \mathbb{T}$.
\begin{align} \re (f(re^{i\theta})) &= \frac{1}{2}\bigl( f(re^{i\theta})) + \ov{f(re^{i\theta})} \bigr)= \re(f(0)) + \frac{1}{2} \sum_{\ell=1}^\infty r^\ell(a_{f}(\ell) e^{i \ell \theta} + \ov{a_{f}(\ell)} e^{-i \ell \theta}) \label{eq.eff-re}
  \end{align} We now separate the proofs of the two clauses in Proposition \ref{prop.fou-tay}.

\paragraph{(Proof of Clause I)}  According to (1.3.18) in \cite{MR2105088}, the following holds for all $e^{i\varphi},e^{i\theta} \in \mathbb{T}$ and all $r \in [0,1)$.
\begin{align*} \mathcal{H}(e^{i\varphi},re^{i\theta}) = 1+2 \sum_{\ell=1}^\infty r^\ell e^{i\ell\theta}e^{-i\ell \varphi} \end{align*}
Combining the last display with (\ref{eq.eff-re}) we obtain:
\begin{align} \re(\mathcal{H}(e^{i\varphi},re^{i\theta})) = 1+ \sum_{\ell=1}^\infty r^\ell\bigl(e^{i\ell\theta} e^{-i\ell \varphi} +e^{-i\ell\theta} e^{i\ell \varphi} \bigr) \label{eq.her-re} \end{align}
Let $\mu \in \prob(\mathbb{T})$, let $r \in (0,1)$ and let $e^{i \theta} \in \bbt$. We compute:
\begin{align} \re(\psi^\mu(re^{i\theta})) & = \int_\mathbb{T} \re(\mathcal{H}(e^{i\varphi},re^{i \theta})) \deee \mu(e^{i\varphi}) \label{eq.psimucomp-1}
\\ & =  1 + \int_{\bbt}\sum_{\ell=1}^\infty r^\ell( e^{i \ell \theta} e^{-i\ell \varphi} + e^{-i\ell \theta} e^{i\ell \varphi})  \deee \mu(e^{i \varphi}) \label{eq.psimucomp-2}
\\ & =  1 + \sum_{\ell=1}^\infty r^\ell \left( e^{i \ell \theta} \int_{\bbt} e^{-i\ell \varphi} \deee \mu(e^{i \varphi}) + e^{-i\ell \theta} \int_{\bbt}  e^{i\ell \varphi}  \deee \mu(e^{i \varphi}) \right) \label{eq.psimucomp-3}
\\ & =   1 + \sum_{\ell=1}^\infty r^\ell \left( e^{i \ell \theta} \hat{\mu}(\ell) + e^{-i\ell \theta} \hat{\mu}(-\ell) \right) \label{eq.psimucomp-4}\end{align}

This computation may be justified as follows.

\begin{itemize}
	\item The equality in (\ref{eq.psimucomp-1}) follows from the definition of $\psi^\mu$ in (\ref{eq.psimu-def}).
	\item (\ref{eq.psimucomp-2}) follows from (\ref{eq.psimucomp-1}) by (\ref{eq.her-re}).
	\item The interchange of the infinite sum and integral in passing from (\ref{eq.psimucomp-3}) to (\ref{eq.psimucomp-2}) is justified since the sum of functions inside the integral converges absolutely.
	\item (\ref{eq.psimucomp-4}) follows from (\ref{eq.psimucomp-3}) by our convention on Fourier coefficients established in (\ref{eq.fou-con}).
\end{itemize}

On the other hand, instantiating (\ref{eq.eff-re}) with the choice $f \coloneqq \psi^\mu$, we find:
\begin{equation*} \re(\psi^\mu(re^{i\theta})) = 1 + \frac{1}{2} \sum_{\ell=1}^\infty r^\ell\left(a_{\psi^\mu}(\ell) e^{i \ell \theta} + \ov{a_{\psi^\mu}(\ell)} e^{-i \ell \theta}\right) \end{equation*}

The above display is a Fourier series for the function $[e^{i\theta} \mapsto \re(\psi^\mu(re^{i \theta}))] : \mathbb{T} \to \mathbb{R}$. The expression in (\ref{eq.psimucomp-4}) is a Fourier series for the same function. Thus the coefficients of these series must agree, and we find that the following holds for all $\ell \in \mathbb{Z}$.
\begin{align*} r^\ell \hat{\mu}(\ell)  = \begin{cases} \frac{1}{2} r^\ell  a_{\psi^\mu}(\ell) & \mbox{ if }\ell > 0
\\ 1 & \mbox{ if } \ell = 0
\\ \frac{1}{2}r^\ell  \ov{a_{\psi^\mu}(\ell)} & \mbox{ if } \ell < 0 \end{cases} \end{align*}
Since $r$ was assumed to be nonzero, this completes the proof of Clause I in Proposition \ref{prop.fou-tay}.

\paragraph{(Proof of Clause II)} Let $\psi \in \operatorname{Cara}(\mathbb{D})$, let $r \in (0,1)$ and let $\ell \in \mathbb{N}$. We compute:
\begin{align} \widehat{\mu_{\psi,r}}(\ell) &= \int_\mathbb{T} e^{-i \ell \theta} \deee \mu_{\psi,r}(e^{i \theta}) \label{eq.mupsicomp-0}
\\ & = \int_\mathbb{T} e^{-i\ell \theta} \re(\psi(re^{i \theta})) \deee \lambda(e^{i \theta}) \label{eq.mupsicomp-1}
\\ & = \int_{\mathbb{T}} e^{-i \ell \theta} \deee \lambda(e^{i \theta}) + \frac{1}{2}  \int_\mathbb{T} \sum_{k=1}^\infty r^k\left(a_{\psi}(k) e^{i(k-\ell)\theta} + \ov{a_{\psi}(k)} e^{-i (k+\ell)\theta} \right)  \deee \lambda(e^{i \theta}) \label{eq.mupsicomp-2}
\\ & = \int_{\mathbb{T}} e^{-i \ell \theta} \deee \lambda(e^{i \theta}) + \frac{1}{2}\sum_{k=1}^\infty r^k\left( a_{\psi}(k)\int_\mathbb{T}  e^{i(k-\ell)\theta}\deee\lambda(e^{i \theta}) + \ov{a_{\psi}(k)}\int_\mathbb{T} e^{-i (k+\ell)\theta}  \deee \lambda(e^{i \theta})\right) \label{eq.mupsicomp-3}
\\ & = \hat{\lambda}(\ell) + \frac{1}{2}\sum_{k=1}^\infty r^k\left( a_{\psi}(k)\,\hat{\lambda}(\ell-k) + \ov{a_{\psi}(k)} \,\hat{\lambda}(k+\ell) \right)  \label{eq.mupsicomp-4} \end{align}
This computation may be justified as follows.
\begin{itemize} \item The equality in (\ref{eq.mupsicomp-0}) follows from our convention on Fourier coefficients established in (\ref{eq.fou-con}).
\item (\ref{eq.mupsicomp-1}) follows from (\ref{eq.mupsicomp-0}) by the definition of $\mu_{\psi,r}$ in (\ref{eq.mupsi-def}).
\item (\ref{eq.mupsicomp-2}) follows from (\ref{eq.mupsicomp-1}) by instantiating (\ref{eq.eff-re}) with $f \coloneqq \psi$.
\item The interchange of the infinite sum and integral in passing from (\ref{eq.mupsicomp-2}) to (\ref{eq.mupsicomp-3}) is justified by absolute convergence of the Taylor series for $\psi$ on compact subsets of $\mathbb{D}$.
\item (\ref{eq.mupsicomp-4}) follows (\ref{eq.mupsicomp-3}) by our convention on Fourier coefficients established in (\ref{eq.fou-con}). \end{itemize}

Since $\ell \in \mathbb{N}$ we have $\hat{\lambda}(k+\ell) = 0$ for all $k \in \mathbb{N}_+$. Moreover, we have:
\begin{align*} \hat{\lambda}(\ell-k)&  =  \begin{cases} 1 & \mbox{ if } k = \ell \\ 0 & \mbox{ if }k \neq \ell  \end{cases}   \end{align*}
From the last two observations and the computation ending in (\ref{eq.mupsicomp-4}) we obtain the following for all $\ell \in \mathbb{N}$.
\begin{align*} \widehat{\mu_{\psi,r}}(\ell) & = \begin{cases} 1 & \mbox{ if }\ell = 0 
\\ \frac{1}{2} r^\ell a_\psi(\ell) & \mbox{ if }\ell >0 \end{cases} \end{align*}
Since $\mu_\psi$ was defined as the vague limit of $(\mu_{\psi,r})_{r \in (0,1)}$ as $r \uparrow 1$, Remark \ref{rem.vague} implies that we have $\hat{\mu_\psi}(\ell)  = \lim_{r \uparrow 1} \hat{\mu_{\psi,r}}(\ell)$ for all $\ell \in \mathbb{Z}$. Combining this with the last display we obtain the following for all $\ell \in \mathbb{N}$.
\begin{align*} \widehat{\mu_{\psi}}(\ell) & = \begin{cases} 1 & \mbox{ if }\ell = 0 
\\ \frac{1}{2} a_\psi(\ell) & \mbox{ if }\ell >0 \end{cases} \end{align*}
Since $\hat{\mu}(-\ell) = \ov{\hat{\mu}(\ell)}$ for all $\ell \in \mathbb{Z}$, this completes the proof of Clause II in Proposition \ref{prop.fou-tay}. \end{proof}

\begin{remark} It is worth noting that the proof of Clause I in Proposition \ref{prop.fou-tay} could be carried out for a single value of $r \in (0,1)$, while the proof of Clause II requires the limit $r \uparrow 1$. \end{remark}

\subsection{Homeomorphic property of the correspondence}

\subsubsection{Preliminary estimates}

Adopt the notation $\mathbb{D}[r]$ for the open disk of radius $r$ around $0$ in the complex plane.

\begin{proposition} \label{prop.geometricbound} Let $f:\mathbb{D} \to \mathbb{C}$ be a holomorphic function such that $|a_f(\ell)| \leq 4$ for all $\ell \in \mathbb{N}$. Also let $N \in \mathbb{N}$ and let $r \in (0,1)$. Then we have:
\begin{align*} \sup_{z \in \mathbb{D}[r]} |f(z)| \leq \left( \sum_{\ell=0}^N |a_f(\ell)| \right) +  \frac{4r^{N+1}}{1-r} \end{align*} \end{proposition}

\begin{proof}[Proof of Proposition \ref{prop.geometricbound}] Let $z \in \mathbb{D}[r]$. Then we have
\begin{align*} \left \vert \sum_{\ell=N+1}^\infty a_f(\ell)z^\ell \right \vert \leq  \sum_{\ell=N+1}^\infty |a_f(\ell)|\, |z|^\ell \leq  4 \sum_{\ell=N+1}^\infty r^\ell = \frac{4r^{N+1}}{1-r} \end{align*}
where the second inequality in the above display follows from our hypothesis that $|a_f(\ell)| \leq 4$ for all $\ell \in \mathbb{N}$.
Since $|z| \leq 1$ we also find
\begin{align*} \left \vert \sum_{\ell=0}^N a_f(\ell)z^\ell \right \vert \leq \sum_{\ell=0}^N |a_f(\ell)| \end{align*}
and combining the two previous displays we obtain:
\begin{align*} \left \vert \sum_{\ell=0}^\infty a_f(\ell)z^\ell \right \vert \leq  \left( \sum_{\ell=0}^N |a_f(\ell)| \right) +  \frac{4r^{N+1}}{1-r} \end{align*}
Since the expression on the left of the last display is $|f(z)|$, this completes the proof of Proposition \ref{prop.geometricbound}. \end{proof}

\begin{proposition} \label{prop.cauchydiff} Let $f:\mathbb{D} \to \mathbb{C}$ be a holomorphic function. Then for all $\ell \in \mathbb{N}_+$ we have:
\begin{align*} |a_f(\ell)| \leq 2 \sup_{z \in \mathbb{D}[2^{-1/\ell}]} |f(z)| \end{align*} \end{proposition}

\begin{proof}[Proof of Proposition \ref{prop.cauchydiff}] Fixing $\ell \in \mathbb{N}_+$, we let $\gamma:[0,2\pi] \to \mathbb{C}$ be the closed simple contour given by $\gamma(t) \coloneqq 2^{-1/\ell} e^{it}$. Writing $f^{(\ell)}$ for the $\ell^{\mrm{th}}$ derivative of $f$, the Cauchy integral formula implies:
\begin{align*} f^{(\ell)}(0) = \frac{\ell!}{2\pi i} \int_0^{2\pi} \frac{f(\gamma(t))}{\gamma(t)^{\ell+1}} \gamma'(t) \deee t \end{align*}
Since $|\gamma(t)| = |\gamma'(t)| = 2^{-1/\ell}$ for all $t \in [0,2\pi]$, we find:
\begin{align*} \left \vert \frac{f(\gamma(t))}{\gamma(t)^{\ell+1}}\gamma'(t) \right \vert \leq 2^{1/\ell} \sup_{z \in \mathbb{D}[2^{-1/\ell}]} |f(z)|  \end{align*}
Combining the two previous displays, we obtain:
\begin{align*} |f^{(\ell)}(0)| \leq \ell!\, 2^{1/\ell} \sup_{z \in \mathbb{D}[2^{-1/\ell}]} |f(z)|  \end{align*}
Since $a_f(\ell) = f^{(\ell)}(0)/\ell!$, this completes the proof of Proposition \ref{prop.cauchydiff}. \end{proof}

\begin{proposition} \label{prop.compunif} The compact uniform topology on $\operatorname{Cara}(\mathbb{D})$ agrees with the topology of pointwise convergence on Taylor coefficients. \end{proposition}

\begin{proof}[Proof of Proposition \ref{prop.compunif}] Let $(\psi_n)_{n \in \mathbb{N}}$ be a sequence of functions in $\operatorname{Cara}(\mathbb{D})$, and let $\psi \in \operatorname{Cara}(\mathbb{D})$. First assume that $\psi_n$ in the $\psi$ in the compact uniform topology. For $\ell \in \mathbb{Z}$ we can use Proposition \ref{prop.cauchydiff} to find:
\begin{align*} |a_{\psi_n}(\ell) - a_\psi(\ell)| = |a_{\psi_n - \psi}(\ell)| \leq 2 \sup_{z \in \mathbb{D}[2^{-1/\ell}]} |\psi_n(z) - \psi(z)| \end{align*}
By hypothesis, the right side of the last display converges to zero by hypothesis as $n \to \infty$, and so we see $\lim_{n \to \infty} a_{\psi_n}(\ell) = a_\psi(\ell)$. Conversely, assume that $\lim_{n \to \infty} a_{\psi_n}(\ell) = a_\psi(\ell)$ for all $\ell \in \mathbb{Z}$. Let $\epsilon > 0$, let $r \in (0,1)$ and choose $N \in \mathbb{N}$ such that $4r^{N+1}/(1-r) \leq \epsilon/2$. Since the functions under hypothesis are Carath\'{e}odory, their Taylor coefficients are bounded by $2$, and so the Taylor coefficients of $\psi_n-\psi$ are bounded by $4$ for all $n \in \mathbb{N}$. Thus we can use Proposition \ref{prop.geometricbound} to find:
\begin{align*} \sup_{z \in \mathbb{D}[r]} |\psi_n(z) - \psi(z)| \leq \left( \sum_{\ell=0}^N |a_{\psi_n-\psi}(\ell)| \right) +  \frac{\epsilon}{2}  \end{align*}
Then if $n \in \mathbb{N}$ is large enough that $|a_{\psi_n}(\ell)-a_{\psi}(\ell)| \leq \epsilon/(2N+2)$ for all $\ell \in \{0,\ldots,N\}$, the last display will show that $|\psi_n(z)-\psi(z)| \leq \epsilon$ for all $z \in \mathbb{D}[r]$. Thus $(\psi_n)_{n \in \mathbb{N}}$ converges to $\psi$ in the compact-uniform topology, as required. \end{proof}

\subsubsection{Establishing the homeomorphism}

\begin{proposition} \label{prop.homeo} Endow the set $\prob(\mathbb{T})$ with the vague topology and the set $\operatorname{Cara}(\mathbb{D})$ with the compact-uniform topology. Then the mappings $[\mu \mapsto \psi^\mu]: \prob(\mathbb{T}) \to \operatorname{Cara}(\mathbb{D})$ and $[\psi \mapsto \mu_\psi]: \operatorname{Cara}(\mathbb{D}) \to \prob(\mathbb{T})$ are homeomorphisms.\end{proposition}

\begin{proof}[Proof of Proposition \ref{prop.homeo}] First, let $(\mu_n)_{n \in \mathbb{N}}$ be a sequence of elements of $\prob(\mathbb{T})$ which converges to $\prob(\mathbb{T})$ in the vague topology. Then Remark \ref{rem.vague} shows that $\lim_{n \to \infty} \hat{\mu_n}(\ell) = \hat{\mu}(\ell)$ for all $\ell \in \mathbb{Z}$, and so Clause I in Proposition \ref{prop.fou-tay} shows that $\lim_{n \to \infty} a_{\psi^{\mu_n}}(\ell) = a_{\psi^\mu}(\ell)$  for all $\ell \in \mathbb{Z}$. Then Proposition \ref{prop.compunif} shows that $\psi^{\mu_n}$ converges to $\psi^\mu$ in the compact-uniform topology, as required.\\
\\
Conversely, let $(\psi_n)_{n \in \mathbb{N}}$ be a sequence of elements of $\operatorname{Cara}(\mathbb{D})$ which converges to $\psi \in \operatorname{Cara}(\mathbb{D})$ in the compact-uniform topology. Then Proposition \ref{prop.compunif} implies that $\lim_{n \to \infty} a_{\psi_n}(\ell) = a_\psi(\ell)$ for all $\ell \in \mathbb{Z}$ and so Proposition \ref{prop.fou-tay} implies that $\lim_{n \to \infty} \hat{\mu_{\psi_n}}(\ell) = \hat{\mu_\psi}(\ell)$ for all $\ell \in \mathbb{Z}$. Then Remark \ref{rem.vague} implies that $(\mu_{\psi_n})_{n \in \mathbb{N}}$ converges to $\mu_\psi$ in the vague topology, as required. \end{proof}

\subsection{Important special cases of the correspondence} \label{sec.special}

It will be useful to understand what the Lebesgue and point measures correspond to under the mapping we have just described. Clause I in Proposition \ref{prop.fou-tay} implies that the following holds for all $\ell \in \mathbb{N}$.
\begin{align} a_{\psi^\lambda}(\ell) = \begin{cases} 1 & \mbox{ if } \ell = 0 \\ 0 & \mbox{ if } \ell \neq 0 \end{cases} \end{align}
Thus $\psi^\lambda(z) = 1$ for all $z \in \mathbb{D}$. \\
\\
In the other extreme, let $\omega \in \mathbb{T}$ and consider the pure point measure $\boldsymbol{\delta}_{\omega} \in \prob(\mathbb{T})$. Using the definition in (\ref{eq.psimu-def}) we find that the following holds for all $z \in \mathbb{D}$.
\begin{align*}
\psi^{\delta_{\omega}}(z) &= \int_{\bbt} \mathcal{H}(\varsigma, z) \deee\boldsymbol{\delta}_{\omega}(\varsigma) = \mathcal{H}(\omega, z) 
\end{align*}
Thus the pure point measure at $\omega$ corresponds to a Herglotz kernel peaking at $\omega$.

\subsection{Exchange of invariant measures and circular functions}

Next, we establish that our correspondence exchanges $\times n$-circular functions and $\times n$-invariant measures. 
\
\begin{proposition} \label{coeffs} We have the following for all $n \in \mathbb{N}$. 
\begin{description} \item[(Clause I)] A measure $\mu \in \prob(\mathbb{T})$ is $\times n$-invariant if and only if $\hat{\mu}_{\psi}(\ell) = \hat{\mu}_{\psi}(n\ell)$ for all $\ell \in \mathbb{Z}$. 
\item[(Clause II)]  A holomorphic function $\psi$ is $\times n$-circular if and only if $a_{\psi}(\ell) = a_{\psi}(n\ell)$ for all $\ell \in \mathbb{Z}$.
\end{description}

\end{proposition}

\begin{proof} Fixing $n \in \mathbb{N}$, we address the two clauses in Proposition \ref{coeffs} separately.

\paragraph{Proof of Clause I:} First, suppose that $\mu$ is $\times n$-invariant and apply the criterion in Definition \ref{invdef} with $f(z)=z^{-\ell}$. We obtain:

\begin{align*} \hat{\mu}(\ell)= \int_{\bbt} z^{-\ell} \deee \mu(z) = \int_{\bbt} (z^n)^{-\ell} \deee \mu{z} = \int_{\bbt} z^{-n\ell} \deee \mu(z) = \hat{\mu}(n\ell) \end{align*}
Conversely, suppose that $\hat{\mu}(\ell) = \hat{\mu}(n\ell)$ for all $\ell \in \mathbb{Z}$. Let $n \in \mathbb{N}$, let $c_{-n},\ldots,c_n \in \mathbb{C}$ and define a Laurent polynomial $p:\mathbb{T} \to \mathbb{C}$ by:
\begin{align*} p(z) \coloneqq \sum_{\ell=-n}^{n} c_{\ell} z^{\ell} \end{align*}
We compute:
\begin{align*}
\int_{\bbt} p(z) \deee \mu(z)  = \int_{\bbt}\sum_{\ell=-n}^{n} c_{\ell} z^{\ell} \deee \mu(z) &= \sum_{\ell=-n}^{n} c_{\ell} \int_{\bbt} z^{\ell} \deee \mu(z) \\
& = \sum_{\ell=-n}^{n} c_{\ell} \hat{\mu}(-\ell) 
= \sum_{\ell=-n}^{n} c_{\ell} \hat{\mu}(-n\ell)  = \sum_{\ell=-n}^{n} c_{\ell} \int_{\bbt} z^{n\ell} \deee \mu(z) = \int_{\bbt} p(z^n) \deee \mu(z)
\end{align*}
According to Corollary 5.4 in Chapter 2 of \cite{MR1970295}, Laurent polynomials are uniformly dense in $C(\bbt)$. Fix $\epsilon>0$, let $f(z) \in C(\bbt)$, and choose $p: \mathbb{T}  \to \mathbb{C}$ to be a Laurent polynomial such that $|f(z) - p(z) | \leq\frac{\epsilon}{2}$ for all $z \in \mathbb{T}$.
Then we have
\begin{align*} \left|\int_{\bbt} f(z) - p(z) d\mu(z)\right| &  \leq  \frac{\varepsilon}{2} &  \left|\int_{\bbt} f(z^n) - p(z^n) d\mu(z)\right| \leq  \frac{\epsilon}{2} \end{align*}
and therefore:
\begin{align*} \left|\int_{\bbt} f(z^n) - f(z) \deee \mu(z)\right| \leq \left|\int_{\bbt} f(z) - p(z) d\mu(z)\right| + \left|\int_{\bbt} p(z^n) - p(z) d\mu(z)\right| +  \left|\int_{\bbt} f(z^n) - p(z^n) d\mu(z)\right| \leq \epsilon \end{align*}
Thus we see that $\mu$ is indeed $\times n$-invariant. This completes the proof of Clause I.

\paragraph{Proof of Clause II:} Let $\psi:\mathbb{D} \to \mathbb{C}$ be a holomorphic function. For all $z \in \mathbb{D}$ we have:
\begin{align} \psi(z^n)= \sum_{\ell=0}^{\infty} a_{\psi}(\ell) z^{n\ell} \label{eq.jop-3} \end{align}
Furthermore, we compute:
\begin{align}
\frac{1}{n} \sum_{k=0}^{n-1} \psi(e^{2\pi i k/n}z) & = \frac{1}{n} \sum_{k=0}^{n-1} \sum_{\ell=0}^{\infty} a_{\psi}(\ell) e^{2\pi i k\ell/n}z^{\ell} \nonumber \\
& = \sum_{\ell=0}^{\infty} a_{\psi}(\ell) \left( \frac{1}{n} \sum_{k=0}^{n-1} e^{2\pi i k\ell/n} \right)z^{\ell} \label{eq.jop-1} \\
& = \sum_{\ell=0}^{\infty} a_{\psi}(n\ell) z^{n \ell} \label{eq.jop-2}
\end{align}
where the last equality follows from the fact that 
\begin{align*} \frac{1}{n} \sum_{k=0}^{n-1} e^{2\pi i k\ell/n} = \begin{cases} 1 & \mbox{if }n\mbox{ divides }\ell \\ 0 & \mbox{else} \end{cases} \end{align*}
Thus $\psi$ is $\times n$-circular if and only if the series in (\ref{eq.jop-3}) agrees with that in (\ref{eq.jop-2}). It is immediate that these series agree if and only if when $a_{\psi}(n\ell)=a_{\psi}(\ell)$ for all $\ell \in \mathbb{N}$, and so the proof of Clause II is complete.
\end{proof}

By combining Proposition \ref{prop.fou-tay} with Proposition \ref{coeffs}, we obtain what for our purposes is the key property of the Herglotz correspondence.

\begin{corollary} \label{cor.coeffs} \begin{description} \item[(Clause I)] A measure $\mu \in \prob(\mathbb{T})$ is $\times n$-invariant if and only if the function $\psi^\mu \in \operatorname{Cara}(\mathbb{D})$ is $\times n$-circular.
\item[(Clause II)] A function $\psi \in \operatorname{Cara}(\mathbb{D})$ is $\times n$-circular if and only if the measure $\mu_\psi \in \prob(\mathbb{T})$ is $\times n$-invariant. \end{description}
	
\end{corollary}

\subsection{Analysis of finitary objects} \label{sec.ooska}

\subsubsection{Finite support and roots of unity}

For $\mu \in \prob(\mathbb{T})$ and $\omega \in \mathbb{T}$ it will be convenient to write $\mu(\omega)$ for the nonnegative number $\mu(\{\omega\})$. 

\begin{proposition} \label{prop.roots} Suppose $\mu \in \prob(\mathbb{T};2,3)$ is finitely supported. Then $\mu$ is supported on the roots of unity. \end{proposition}

\begin{proof}[Proof of Proposition \ref{prop.roots}]  Write $S$ for the support of $\mu$, so we have:
\begin{equation*} S = \{\omega \in \mathbb{T}: \mu(\omega) > 0\} \end{equation*} 
Let $\omega \in S$ and write $\omega = e^{i\theta}$ for some $\theta \in \mathbb{R}$. Since $\mu$ was assumed to be $\times (2,3)$-invariant we have $\mu(e^{i2^k\theta}) = \mu(e^{i\theta})$ for all $k \in \mathbb{N}$, and hence $e^{i2^k\theta} \in S$ for all $k \in \mathbb{N}$.\\
\\
The last assertion implies that the set $\{e^{i2^k\theta}:k \in \mathbb{N}\}$ is finite. Letting $j,k \in \mathbb{N}$ be distinct numbers such that $e^{i2^j\theta} = e^{i2^k\theta}$, we find $e^{i(2^j-2^k)\theta} = 1$. Using the assumption that $k \neq \ell$, we can choose $d = 2^j-2^k$ to find that $\omega^d = 1$. This completes the proof of Proposition \ref{prop.roots}.\end{proof}

\subsubsection{Ergodic finitely supported measures}

For a point $\omega \in \mathbb{T}$ let $\mathcal{O}(\omega)$ denote the orbit of $\omega$ under the $\times (2,3)$ semigroup, given explicitly as follows.
\begin{align*} \mathcal{O}(\omega) \coloneqq \{\omega^{2^j3^k}: j,k \in \mathbb{N}  \} \end{align*}

\begin{proposition} \label{prop.uniformorbit} Let $\mu \in \mathrm{Prob}(\mathbb{T}; 2,3)$ be finitely supported and ergodic. Then there exists a root of unity $\omega \in \mathbb{T}$ such that:
\begin{align*} \mu = \frac{1}{|\mathcal{O}(\omega)|} \sum_{\vartheta \in \mathcal{O}(\omega)} \boldsymbol{\delta}_\vartheta \end{align*} \end{proposition}

\begin{proof}[Proof of Proposition \ref{prop.uniformorbit}] Proposition \ref{prop.roots} implies that $\operatorname{supp}(\mu)$ is a disjoint union of orbits of roots of unity. Let $m$ be the number of distinct orbits contained in the support of $\mu$, and enumerate these orbits as $\mathcal{O}_1,\ldots, \mathcal{O}_m$. It is immediate that $\mu(\omega) = \mu(\vartheta)$ for all $\ell \in \{1,\ldots,m\}$ and all $\omega,\vartheta \in \mathcal{O}_\ell$. Thus in order to complete the proof of Proposition \ref{prop.uniformorbit} it suffices to show that $m=1$. Indeed, we have
\begin{align} \mu = \sum_{\ell=1}^m \frac{1}{|\mathcal{O}_\ell|} \sum_{\vartheta \in \mathcal{O}_\ell} \boldsymbol{\delta}_\vartheta \end{align}
and so ergodicity of $\mu$ implies that $m=1$, as required. \end{proof}

\subsubsection{Maximal roots of finitely supported measures}

Proposition \ref{prop.roots} justifies the following definition.

\begin{definition} \label{def.maxr} Let $\mu \in \prob(\mathbb{T};2,3)$ be finitely supported. We define the \emph{maximal root of} $\mu$ to be the maximal natural number $d$ such that the support of $\mu$ contains a $d^{\mrm{th}}$ root a unity. We denote the maximal root of $\mu$ by $\operatorname{maxr}(\mu)$. \end{definition}

\begin{proposition} \label{prop.maxr} Let $\mu \in \prob(\mathbb{T};2,3)$ be finitely supported and ergodic. Then we have:
\begin{align*}   \lvert \operatorname{supp}(\mu) \rvert \leq  \operatorname{rmax}(\mu) \leq  6^{\lvert \operatorname{supp}(\mu) \rvert} \end{align*} \end{proposition}

\begin{proof}[Proof of Proposition \ref{prop.maxr}] By definition, we have that $\operatorname{supp}(\mu)$ is a subset of the $\operatorname{rmax(\mu)}^{\mrm{th}}$ roots of unity, and so we find  $\lvert \operatorname{supp}(\mu) \rvert \leq  \operatorname{rmax}(\mu)$. On the other hand, Proposition \ref{prop.uniformorbit} implies that $\operatorname{supp}(\mu) = \mathcal{O}(\omega)$ for some root of unity $\omega$. Writing $r \coloneqq \operatorname{rmax}(\mu)$, we may assume that $\omega = e^{2\pi i m/r}$ for some $m \in \{0,\ldots,r-1\}$ with $\operatorname{gcd}(m,r) =1$.\\ 
\\
Now, writing $n \coloneqq \lvert \operatorname{supp}(\mu)\rvert$, there exist $j,k \in \mathbb{N}$ with $0 < jk \leq n$ such that $\omega^{2^j3^k} = \omega$. This implies $e^{2\pi i 2^j3^k m/r} = e^{2\pi i m/r}$ and so $e^{2\pi i (2^j3^k-1) m/r} =1$. Thus find $(2^j3^k-1) m/r \in \mathbb{Z}$ or equivalently $2^j3^k m \equiv m \,\,(\operatorname{mod}{r})$. Since we have assumed $\operatorname{gcd}(m,r) =1$, we can can divide both sides of this congruence by $m$ to find $2^j3^k \equiv 1 \,\,(\operatorname{mod}{r})$. This implies in particular that $r \leq 2^j3^k$, and since $jk \leq n$ we find $r \leq 6^n$. This completes the proof of Proposition \ref{prop.maxr}. \end{proof}

\subsection{Convexity results}

\subsubsection{Exchange of convex combinations}

The following is an immediate consequence of the definition of the Herglotz correspondence.

\begin{proposition} \label{prop.convexity} \begin{description} \item[(Clause I)] For all $\mu,\nu \in \prob(\mathbb{T})$ and all $t \in [0,1]$ we have $\psi^{t\mu + (1-t)\nu} = t \psi^\mu + (1-t)\psi^\nu$.
\item[(Clause II)] For all $\psi,\phi \in \operatorname{Cara}(\mathbb{D})$ and all $t \in [0,1]$ we have $\mu_{t\psi+(1-t)\phi} = t \mu_\psi + (1-t)\mu_\phi$. \end{description} \end{proposition}

The following is an immediate consequence of Proposition \ref{prop.convexity}.

\begin{corollary} \label{cor.ergext} \begin{description} \item[(Clause I)] Let $\mu \in \prob(\mathbb{T}; 2,3)$. Then $\psi^\mu$ is extreme if and only if $\mu$ is ergodic. 
\item[(Clause II)] Let $\psi \in \operatorname{Cara}(\mathbb{D}; 2,3)$. Then $\psi$ is extreme if and only if $\mu_\psi$ is ergodic. \end{description}\end{corollary}

\subsubsection{Extreme rational \cara functions}

\begin{proposition} \label{prop.singular} Suppose that $\psi \in \operatorname{Cara}(\mathbb{D};2,3)$ is extreme and nonconstant. Then $\lambda$ and $\mu_\psi$ are mutually singular elements of $\prob(\mathbb{T})$. \end{proposition}

\begin{proof}[Proof of Proposition \ref{prop.singular}] Since $\psi$ is nonconstant, we have $\lambda \neq \mu_\psi$. Hence there exists a continuous function $f:\mathbb{T} \to \mathbb{R}$ such that:
\begin{align*} \int_\mathbb{T} f(\omega) \deee \lambda(\omega) \neq \int_{\mathbb{T}} f(\omega) \deee \mu_\psi(\omega) \end{align*}
We will now use the pointwise ergodic theorem for actions of the semigroup $\mathbb{N}^2$. This appears, for example, as Theorem 1.1 in \cite{MR797411}. Since $\lambda$ is ergodic, there exists a Borel subset $L$ of $\mathbb{T}$ with $\lambda(L) = 1$ such that the following holds for all $\omega \in L$.
\begin{align*} \lim_{n \to \infty} \frac{1}{n} \sum_{j,k = 0}^{n-1} f(2^j3^k\omega) =  \int_\mathbb{T} f(\omega) \deee \lambda(\omega)  \end{align*}
Similarly, Clause II in Corollary \ref{cor.ergext} implies that $\mu_\psi$ is ergodic and so there exists a Borel subset $M$ of $\mathbb{T}$ with $\mu_\psi(M) = 1$ such that the following holds for all $\omega \in M$.
\begin{align*} \lim_{n \to \infty} \frac{1}{n} \sum_{j,k = 0}^{n-1} f(2^j3^k\omega) =  \int_{\mathbb{T}} f(\omega) \deee \mu_\psi(\omega)   \end{align*}
Combining the three previous displays, we find that $L$ and $M$ must be disjoint. Thus $\lambda$ and $\mu_\psi$ are mutually singular, as required. \end{proof}

\begin{proposition} \label{prop.wellroot} Suppose that $\psi \in \operatorname{Cara}(\mathbb{D};2,3)$ is rational, nonconstant and extreme. Then there exists a root of unity $\omega \in \mathbb{T}$ such that the following holds for all $z \in \mathbb{D}$. \begin{align} \psi(z) = \frac{1}{|\mathcal{O}(\omega)|} \sum_{\vartheta \in \mathcal{O}(\omega)} \mathcal{H}(\vartheta, z) \label{eq.niceform} \end{align}\end{proposition}

\begin{proof} Letting $\psi$ be as in the statement of Proposition \ref{prop.wellroot}, we first show that $\mu_\psi$ is finitely supported. Since $\psi$ is nonconstant and extreme, Proposition \ref{prop.singular} implies that $\mu_\psi$ is mutually singular with $\lambda$. Therefore Item (iv) in Section 3 of Chapter 1.3 on page 29 of \cite{MR2105088} implies the following.
\begin{align} \mu_\psi\!\left(\!\left\{e^{i\theta} \in \mathbb{T}: \lim_{r \uparrow 1} \psi_\mu(re^{i\theta}) = \infty \right\}\!\right) = 1\end{align}
Since $\psi$ is rational, the above display implies that $\mu_\psi$ is finitely supported. Passing Proposition \ref{prop.uniformorbit} back through the correspondence completes the proof of Proposition \ref{prop.wellroot}.\end{proof}

\subsection{Main argument for equivalence of ergodic and complex-analytic conjectures}

\paragraph{Deduction of Conjecture \ref{conj.erg-1} from Conjecture \ref{conj.comp-1}} Assume Conjecture \ref{conj.comp-1} and let $\mu \in \prob(\mathbb{T}; 2,3)$ be ergodic. According to Clause I in Corollary \ref{cor.coeffs}, $\psi^\mu$ is an element of $\operatorname{Cara}(\mathbb{D}; 2,3)$, and according to Corollary \ref{cor.ergext} we have that $\psi^\mu$ is extreme. Thus Conjecture \ref{conj.comp-1} implies that $\psi^\mu$ is rational. If $\psi^\mu$ is constant, then the discussion in Section \ref{sec.special} implies that $\mu = \lambda$. On other hand, if $\psi^\mu$ is nonconstant then it has the form appearing in (\ref{eq.niceform}), and the discussion in Section \ref{sec.special} implies that $\mu$ is the uniform measure in $\mathcal{O}(\omega)$ for some root of unity $\omega$. Thus we find that Conjecture \ref{conj.comp-1} implies Conjecture \ref{conj.erg-1}, as required.

\paragraph{Deduction of Conjecture \ref{conj.comp-1} from Conjecture \ref{conj.erg-1}} Assume Conjecture \ref{conj.erg-1}, and let $\psi \in \operatorname{Cara}(\mathbb{D};2,3)$. According to Clause II in Proposition \ref{coeffs}, $\mu_\psi$ is an element of $\prob(\mathbb{T};2,3)$ and according to Corollary \ref{cor.ergext} we have that $\mu_\psi$ is ergodic. If $\mu_\psi$ is infinitely supported, then Conjecture \ref{conj.erg-1} implies that $\mu_\psi = \lambda$, and so the discussion in Section \ref{sec.special} implies that $\psi$ is the constant function $1$. On the other hand, if $\mu$ is finitely supported then Proposition \ref{prop.uniformorbit} implies that $\mu$ is the uniform measure on $\mathcal{O}(\omega)$ for some root of unity $\omega$. Thus the discussion in Section \ref{sec.special} implies that $\psi$ is a finite sum of Herglotz kernels. In either case we find that $\psi$ is rational, as required.

\paragraph{Deduction of Conjecture \ref{conj.erg-2} from Conjecture \ref{conj.comp-2}} Assume Conjecture \ref{conj.comp-2} and let $\mu \in \prob(\mathbb{T};2,3)$ be arbitrary. By hypothesis we can find a sequence $(\psi_n)_{n \in \mathbb{N}}$ of rational elements of $\prob(\mathbb{T};2,3)$ which converges to $\psi^\mu$ in the compact-uniform topology. We now distinguish two cases.

\begin{itemize} \item In this case we assume that $\psi_n$ is constant for infinitely values of $n$. Then $\psi$ must also be constant, and so we see that $\mu = \lambda$. We can realize $\lambda$ as the vague limit of the uniform measures on the $p_n^{\mrm{th}}$ roots of unity, where $p_n$ is the $(n+2)^{\mrm{th}}$ prime number. Each of these measures is $\times (2,3)$-invariant, and so we see that $\mu$ is indeed a vague limit of finitely supported measures.
\item In this case we assume that $\psi_n$ is nonconstant for all but finitely values of $n$. Decomposing $\psi_n$ into a finite sum of extreme rational elements of $\operatorname{Cara}(\mathbb{D};2,3)$, we can use Proposition \ref{prop.wellroot} to see that $\mu_{\psi_n} \in \prob(\mathbb{T};2,3)$ is finitely supported for all $n \in \mathbb{N}$. Since $(\psi_n)_{n \in \mathbb{N}}$ was assumed to converge to $\psi_\mu$ in the compact-uniform topology, Proposition \ref{prop.homeo} implies that $(\mu_{\psi_n})_{n \in \mathbb{N}}$ converges to $\mu$, as required. \end{itemize}

\paragraph{Deduction of Conjecture \ref{conj.comp-2} from Conjecture \ref{conj.erg-2}} Assume Conjecture \ref{conj.erg-2} and let $\psi \in \operatorname{Cara}(\mathbb{D};2,3)$. By hypothesis we can find a sequence $(\mu_n)_{n \in \mathbb{N}}$ of finitely supported elements of $\prob(\mathbb{T};2,3)$ which converges vaguely to $\mu_\psi$. Then Proposition \ref{prop.homeo} implies that the sequence $(\psi^{\mu_n})_{n \in \mathbb{N}}$ of elements of $\operatorname{Cara}(\mathbb{D};2,3)$ converges in the compact-uniform topology to $\psi$ and the discussion in Section \ref{sec.special} implies that every element of this sequence os rational, as required.

\paragraph{Deduction of Conjecture \ref{conj.erg-3} from Conjecture \ref{conj.comp-3}} Assuming Conjecture \ref{conj.comp-3} and let $(\mu_n)_{n \in \mathbb{N}}$ be a sequence of ergodic finitely supported elements of $\prob(\mathbb{T};2,3)$ such that $\lim_{n \to \infty} \lvert \operatorname{supp}(\mu_n) \rvert = \infty$. Then Proposition \ref{prop.wellroot} implies that $\psi^{\mu_n}$ is an extreme rational element of $\operatorname{Cara}(\mathbb{D};2,3)$ for all $n \in \mathbb{N}$, and the number of poles of $\psi^{\mu_n}$ is equal to the number of points in the support of $\mu_n$. Thus Conjecture \ref{conj.comp-3} implies that $\psi^{\mu_n}$ converges compact-uniformly to the constant function $1$ as $n \to \infty$, and so Proposition \ref{prop.homeo} implies that $\mu_n$ converges vaguely to $\lambda$, as required.

\paragraph{Deduction of Conjecture \ref{conj.comp-3} from Conjecture \ref{conj.erg-3}} Let $(\psi_n)_{n \in \mathbb{N}}$ be a sequence of extreme rational elements of $\operatorname{Cara}(\mathbb{D};2,3)$ and assume that the number of poles of $\psi_n$ diverges to infinity. Then Proposition \ref{prop.wellroot} implies that the element $\mu_{\psi_n}$ of $\prob(\mathbb{T};2,3)$ is a finitely supported measure, and that $\lvert \operatorname{supp}(\mu_{\psi_n}) \rvert$ is equal to the number of poles of $\psi_n$. Thus Conjecture \ref{conj.erg-3} implies that $\mu_{\psi_n}$ converges vaguely to $\lambda$ as $n \to \infty$, and so Proposition \ref{prop.homeo} implies that $\psi_n$ converges compact-uniformly to the constant function $1$, as required.

\section{Measure-tracial state correspondence} \label{sec.tracial}

In this section we establish a correspondence between elements of $\operatorname{Prob}(\mathbb{T};2,3)$ and elements of $\operatorname{TSt}(\mathfrak{F}_{2,3})$. Unlike the correspondence described in Section \ref{sec.meas-fun}, this will not be a bijective correspondence. (See Remark \ref{rem.partinv} below.) However, it will preserve enough information to allow us to show the equivalence of Conjectures \ref{conj.erg-1} and \ref{conj.see-1}, the equivalence of Conjectures \ref{conj.erg-2} and \ref{conj.see-2}, and the equivalence of Conjectures \ref{conj.erg-3} and \ref{conj.see-3}.\\
\\
Note that if $\mu$ were a $\times (n_1,\ldots,n_m)$-invariant measure for some $n_1,\ldots,n_m \in \mathbb{N}$, we could obtain a tracial state on the full group $\Cstar$-algebra of $ \bbz^m \ltimes \bbz[1/(n_1 \cdots n_m)]$ by following the same method we describe.
 
\subsection{Introducing the tracial correspondence}

\subsubsection{Generalities on the GNS construction} \label{subsec.char}

In the present Subsection \ref{subsec.char}, we exposit a standard method for constructing positive linear functionals on group $C^\ast$-algebras. If $G$ is a group we will write $\iota_G$ for the identity element of $G$.

\begin{definition} \label{def.character}
    Let $G$ be a countable group and $\phi$ a complex-valued function on $G$. 
    \begin{itemize}
       \item We define $\phi$ to be \emph{positive definite} if for every finite subset $F$ of $G$ and every function $c: F \rightarrow \mathbb{C}$ we have:
       \begin{align*} \sum_{g,h \in F} c(g) \phi(h^{-1}g)\overline{c(h)} \geq 0 \end{align*}
       \item We define $\chi$ to be a \emph{character} if it is positive definite and invariant under conjugation in the sense that $\phi(h^{-1}gh) = \phi(g)$ for all $g,h \in G$.
       \item We define $\phi$ to be \emph{normalized} if $\chi(\iota_G) = 1$.
    \end{itemize}
\end{definition}

We also recall the following version of the GNS construction. Here, for $g \in G$ we let $\tau_g$ denote the permutation of $G$ given by the left-translation action $\tau_g(h) = g^{-1}h$. 

\begin{proposition}[Proposition 1.B.8 in \cite{MR4249450}]  \label{prop.posdef} Let $G$ be a countable group and let $\phi$ be a positive definite function on $G$. Then there exists a Hilbert space $\mathcal{H}_\phi$ and a function $\Omega_\phi:G \to \mathcal{H}_\phi$ such that $\langle \Omega_\phi(g),\Omega_\phi(h) \rangle = \phi(h^{-1}g)$ for all $g,h \in G$ and such that the linear span of the set $\{\Omega_\phi(g):g \in G\}$ is dense in $\mathcal{H}_\phi$. Furthermore, for any $g \in G$ there exists a unique unitary `translation' operator $\eusc{T}_\phi(g) \in \mrm{Un}(\mathcal{H}_\phi)$ which which implements the bottom arrow in the following commutative diagram. \[   \begin{tikzcd} G \arrow[ddd,"\Omega_\phi" left]  \arrow[rrr,"\tau_g"]  &&&   G \arrow[ddd,"\Omega_\phi"]  \\ \\ \\  \mathcal{H}_\phi \arrow[rrr,"\eusc{T}_\phi(g)"] &&& \mathcal{H}_\phi \end{tikzcd}  \]

 \end{proposition}
 
 Since the map $g \mapsto \tau_g$ is a homomorphism from $G$ into the group of permutations of $G$, it follows that the map $\eusc{T}_\phi: G \to \mrm{Un}(\mathcal{H}_\phi)$ is a unitary representation of $G$. According to the universal property of $C^\ast(G)$, there exists a unique $\ast$-homomorphism $\eusc{A}_\phi:C^\ast(G) \to \mrm{End}(\mathcal{H}_\phi)$ such that the following diagram commutes.
 \[ \begin{tikzcd} & G  \arrow[ddl, "\upsilon_G" above left]  \arrow[ddr, "\eusc{T}_\phi" above right]  & \\  \\ C^\ast(G) \arrow[rr,"\eusc{A}_\phi"] & & \mrm{End}(\mathcal{H})  \end{tikzcd} \] 
 
 Now, assume in addition that the positive definite function $\phi$ is normalized. We define a linear functional $\Phi:C^\ast(G) \to \mathbb{C}$ by setting $\Phi(s) = \langle \eusc{A}_\phi(s)\Omega_\Phi(\iota_G), \Omega_\Phi(\iota_G) \rangle$ for an operator $s \in C^\ast(G)$. The construction guarantees that $\Phi$ is a state on the algebra $C^\ast(G)$ extending $\phi$. More explicitly, this means that $\Phi$ is a linear functional on $C^\ast(G)$ such that $\Phi(s^\ast s) \geq 0$ for all $s \in C^\ast(G)$ and such that the following diagram commutes.  \[  \begin{tikzcd} \arrow[ddr, "\phi" below left] G \arrow[rr, "\upsilon_G" ] && C^\ast (G)  \arrow[ddl,"\Phi", start anchor={[xshift=-0.5ex]}  ] \\ \\ &\mathbb{C} &    \end{tikzcd} \]
 
 It is straightforward to verify that if we further assume $\phi$ is a character of $G$, then the state $\Phi$ is tracial.

\begin{remark} \label{rem.stchar} The preceding discussion in Section \ref{subsec.char} implies that there is a one-to-one correspondence between tracial states on $C^\ast(G)$ and characters on $G$. If a tracial state $\Phi \in \operatorname{TSt}(C^\ast(G))$ is given, we will sometimes abuse notation by conflating $\Phi$ with its associated character on $G$. Thus, for example, we may write $\mathcal{H}_\Phi$ for the underlying Hilbert space of the GNS representation of $C^\ast(G)$ associated with $\Phi$. We will further abuse notation by identifying $G$ with a subset of $C^\ast(G)$ via the map $\upsilon_G$. Thus, for example, if $g \in G$ we may write $\Phi(g)$ to denote the complex number $\Phi(\upsilon_G(g))$.   \end{remark}

\subsubsection{Characters on $\mathfrak{F}_{2,3}$ and invariant measures} \label{subsec.extend}

Now, fix $\mu \in \prob(\mathbb{T};2,3)$, in order to define a state $\Phi^\mu \in \operatorname{TSt}(\mathfrak{F}_{2,3})$. First, consider the function  $\phi^\mu : \bbz \rightarrow \mathbb{C}$ given by $\phi(\ell) \coloneqq \hat{\mu}(\ell)$. By Proposition \ref{coeffs}, $\phi^\mu$ has the invariance property $\phi^\mu(\ell)=\phi^\mu(2\ell)=\phi^\mu(3\ell)$ for all $\ell \in \bbz$. Moreover, by the Herglotz-Bochner theorem, $\phi$ is positive definite. We extend $\phi^\mu$ to a function $\mathcal{E}^\mu:\bbz[1/6] \to \mathbb{C}$ by setting $\mathcal{E}^\mu(m6^{-n}) = \phi(m)$ for $m \in \mathbb{Z}$ and $n \in \mathbb{N}$. The function $\mathcal{E}^\mu$ is well-defined because of the invariance property of $\phi^\mu$. We further extend $\mathcal{E}^\mu$ to a function $\chi^\mu:\fgroup \to \mathbb{C}$ as follows. 
\begin{align*} \chi^\mu (j,k;p) = \begin{cases} \mathcal{E}^\mu(p) & \textrm{ if } j=k=0 \\ 0 & \textrm{else}\end{cases} \end{align*}

We now show that the above construction is compatible with Definition \ref{def.character}.

\begin{proposition} \label{prop.character}
For all $\mu \in \prob(\mathbb{T};2,3)$ we have that $\chi^\mu$ is a character on $\mathfrak{F}_{2,3}$.
\end{proposition}

\begin{proof}[Proof of Proposition \ref{prop.character}]
We begin by establishing that $\mathcal{E}^\mu$ is positive definite on $\bbz[1/6]$, following Proposition 4.1 of \cite{MR3679610}. Fix $q_1, ..., q_n \in \bbz[1/6]$ and $c_1, ..., c_n \in \mathbb{C}$. Choose $m \in \mathbb{N}$ so that $6^m (q_s-q_t)$ is an integer for all $s,t \in \{1,...,n\}$. By construction,  
\[ \mathcal{E}^\mu(q_s - q_t) = \phi(6^mq_s - 6^m q_t) \]

Then, since $\phi$ is positive definite we have:
\begin{align*}
    \sum_{s,t =1}^{n} c_s \overline{c_t} \mathcal{E}^\mu(q_s - q_t) = \sum_{s,t =1}^{n} c_s \overline{c_t} \phi^\mu(6^mq_s - 6^mq_t) \geq 0 .
\end{align*}
Since $\chi^\mu$ vanishes off the normal subgroup $\bbz[1/6]$ of  $\fgroup$, positive definiteness of $\mathcal{E}^\mu$ implies positive definiteness of $\chi^\mu$. Conjugation-invariance of $\chi^\mu$ follows by direct computation:
\begin{align*} \chi^\mu\bigl((j,k;p) \star (\ell,m;q) \star (-j,-k; -2^{-j}3^{-k}p)\bigr) & = \chi^\mu (\ell,m;p + 2^j3^k q - 2^\ell3^m p) \\
& = \begin{cases} \mathcal{E}^\mu (2^j3^kq) & \textrm{if } \ell=m=0 \\ 0 & \textrm{else}\end{cases} \\
& = \begin{cases} \mathcal{E}^\mu (q) & \textrm{if } \ell=m=0 \\ 0 & \textrm{else}\end{cases} \\
& = \chi^\mu(\ell,m;q) \end{align*}
Thus the proof of Proposition \ref{prop.character} is complete.
\end{proof}

In light of Proposition \ref{prop.character}, we can define $\Phi^\mu$ to be the tracial state on $C^\ast(\mathfrak{F}_{2,3})$ generated by the construction in Section \ref{subsec.char}. For the reverse direction of the correspondence, given $\Phi \in \operatorname{TSt}(C^\ast(\mathfrak{F}_{2,3}))$ we define $\mu_\Phi \in \prob(\mathbb{T};2,3)$ by the stipulation that $\hat{\mu_\Phi}(\ell) = \Phi(0,0;\ell)$ for all $\ell \in \mathbb{Z}$.

\begin{remark} \label{rem.partinv} Although the map $[\Phi \mapsto \mu_\Phi]:\prob(\mathbb{T};2,3) \to \operatorname{TSt}(C^\ast(\mathfrak{F}_{2,3}))$ is not injective, it follows immediately from the definition that the opposite map $[\mu \mapsto \Phi^\mu]:\prob(\mathbb{T};2,3) \to \operatorname{TSt}(C^\ast(\mathfrak{F}_{2,3}))$ is a partial inverse in the sense that $\mu_{\Phi^\mu} = \mu$ for all $\mu \in \prob(\mathbb{T};2,3)$. \end{remark}

\subsection{Quantitative properties of the tracial correspondence}

\subsubsection{Preliminaries} \label{sec.prelimtst}

In Section \ref{sec.prelimtst} we establish three preliminary results which will be useful to control supports of measures and conditional dimensions of states.

\begin{proposition} \label{prop.zeros} Let $\mu \in \mathrm{Prob}(\mathbb{T})$, let $d \in \mathbb{N}$ and suppose there exist complex numbers $c_0,\ldots,c_d$, not all zero, such that the following holds.
\begin{align} \label{eq.zeros-1} \sum_{\ell,m = 0}^d c_\ell \ov{c_m} \,\hat{\mu}(\ell-m) =0 \end{align}
Then we have $\lvert \operatorname{supp}(\mu) \rvert \leq d$. \end{proposition}

\begin{proof}[Proof of Proposition \ref{prop.zeros}] Fix $d \in \mathbb{N}$ and $c_0,\ldots,c_d \in \mathbb{C}$. We have:
\begin{align} \sum_{\ell,m = 0}^d c_\ell \ov{c_m} \,\hat{\mu}(\ell-m) = \sum_{\ell,m = 0}^d c_\ell \ov{c_m}  \int_\mathbb{T} z^{m-\ell} \deee \mu(z) = \int_\mathbb{T} \left \vert \sum_{m = 0}^d \ov{c_m}z^{m} \right \vert^2\!\! \deee \mu(z) \label{eq.zeroes-4} \end{align}
where the second inequality in the last display follows by the square inside the integral. Now, define a subset $\Xi$ of $\mathbb{C}$ by:
\begin{equation*} \Xi \coloneqq \left \{ z \in \mathbb{C}:  \sum_{m = 0}^d \ov{c_m} z^{m} = 0 \right \} \end{equation*}
Assuming $c_0,\ldots,c_d$ are not all zero, we see that $\Xi$ is contained in the vanishing set of a nonzero polynomial with degree at most $d$, and so the fundamental theorem of algebra implies $|\Xi| \leq d$. If (\ref{eq.zeros-1}) holds then (\ref{eq.zeroes-4}) implies that $\mu(\Xi) =1$ and so we must have $\lvert \operatorname{supp}(\mu) \rvert \leq d$, as required. \end{proof}

\begin{proposition} \label{prop.zeroop} Let $\Phi \in \operatorname{TSt}(C^\ast(\mathfrak{F}_{2,3}))$, let $\mathcal{H}_\Phi$ be the underlying Hilbert space of the GNS representation of $C^\ast(\mathfrak{F}_{2,3})$ associated with $\Phi$, and let $\eusc{A}_\Phi:C^\ast(\mathfrak{F}_{2,3}) \to \operatorname{End}(\mathcal{H})$ and $\Omega_\Phi:\mathfrak{F}_{2,3} \to \mathcal{H}$ be as in Proposition \ref{prop.posdef}. Also let $d \in \mathbb{N}$ and let $c_0,\ldots,c_d \in \mathbb{C}$. Writing $0_\mathcal{H}$ for the zero vector in $\mathcal{H}$, we have following equivalence.
\begin{align} \label{eq.zeroop-1} \sum_{\ell,m = 0}^d c_\ell \ov{c_m} \Phi(0,0;\ell-m) =0 \iff  \left( \sum_{m=0}^d  c_m \eusc{A}_\Phi(0,0;m) \right) \Omega_\Phi(0,0;0) = 0_\mathcal{H} \end{align} \end{proposition}

\begin{proof}[Proof of Proposition \ref{prop.zeroop}] Let $d \in \mathbb{N}$ and let $c_0,\ldots,c_d \in \mathbb{C}$ throughout the proof of Proposition \ref{prop.zeroop}. By the construction in Proposition \ref{prop.posdef},  the following holds for all $k,l \in \mathbb{Z}$ and all $q \in \mathbb{Z}[1/6]$.
\begin{align*} \Phi(0,0;\ell-m) & = \langle \Omega_\Phi(0,0;\ell+q), \Omega_\Phi(0,0;m+q) \rangle  = \bigl \langle \eusc{A}_\Phi(0,0;\ell) \Omega_\Phi(0,0;0),\eusc{A}_\Phi(0,0;m)  \Omega_\Phi(0,0;0) \bigr \rangle \label{eq.zeroop-2} \end{align*}
The above display implies:
\begin{align*}  \sum_{\ell,m = 0}^d c_\ell \ov{c_m} \Phi(0,0;\ell-m) = \nml \left( \sum_{m=0}^d c_m \eusc{A}_\Phi(0,0;m) \right) \Omega_\Phi(0,0;0) \nmr^2 \end{align*}
This completes the proof of Proposition \ref{prop.zeroop}. \end{proof}

\begin{proposition} \label{prop.equalsone} Let $\Phi \in \operatorname{TSt}(C^\ast(\mathfrak{F}_{2,3}))$ and assume there exist $j,k \in \mathbb{N}$ and $r \in \mathbb{Z}$ such that $\Phi(0,0;r2^{-j}3^{-k}) = 1$. Then we have $\operatorname{cndm}(\Phi) \leq r$.\end{proposition}

\begin{proof}[Proof of Proposition \ref{prop.equalsone}] It is elementary that if $G$ is a discrete group and $\chi:G \to \mathbb{C}$ is a normalized character, then the subset $\{g \in G:\chi(g) = 1\}$ of $G$ must be a normal subgroup. For all $q \in \mathbb{Z}[1/6]$ we have that the element $(0,0;rq)$ of $\mathfrak{F}_{2,3}$ lies in the normal subgroup generated by $(0,0;r2^{-j}3^{-k})$, and so we have $\Phi(0,0;r) =1$ for all $q \in \mathbb{Z}[1/6]$. \\
\\
Now, let $\mathcal{H}_\Phi$ be the underlying Hilbert space of the GNS representation of $C^\ast(\mathfrak{F}_{2,3})$, let $\eusc{A}_\Phi:\mathfrak{F}_{2,3} \to \operatorname{End}(\mathcal{H}_\Phi)$ be the canonical $\ast$-homomorphism and let $\Omega_\Phi:\mathfrak{F}_{2,3} \to \mathcal{H}_\Phi$ be the canonical embedding. Again let $j,k \in \mathbb{Z}$ and let $p,q \in \mathbb{Z}[1/6]$. We have:
\begin{align*} \Bigl \langle \eusc{A}_\Phi(0,0;rp) \Omega_\Phi(j,k;q)  , \Omega_\Phi(j,k;q)  \Bigr \rangle & = \Bigl \langle \eusc{A}_\Phi(j,k;q)^{-1}\eusc{A}_\Phi(0,0;rp)\eusc{A}_\Phi(j,k;q)  \Omega_\Phi(0,0;0)  , \Omega_\Phi(0,0;0)  \Bigr \rangle 
\\ & =  \Bigl \langle \eusc{A}_\Phi(0,0;2^j3^krp)  \Omega_\Phi(0,0;0)  , \Omega_\Phi(0,0;0)  \Bigr \rangle
\\ & =  \Phi(0,0;2^j3^krp) 
\\ & = 1 \end{align*}

Since $||\Omega_\Phi(j,k;q)|| = 1$, we find that $\eusc{A}_\Phi(0,0;rp)\Omega_\Phi(j,k;q) = \Omega_\Phi(j,k;q)$. Therefore the operator $\eusc{A}_\Phi(0,0;rp)$ fixes every element of the basis $\{\Omega_\Phi(j,k;q): j,k \in \mathbb{Z}, q \in \mathbb{Z}[1/6]\}$ for $\mathcal{H}_\Phi$, and we find that $\eusc{A}_\Phi(0,0;rp)$ must be the identity operator on $\mathcal{H}_\Phi$. It follows that $\eusc{A}_\Phi$ maps the subgroup $r\mathbb{Z}[1/6]$ of $\mathbb{Z}[1/6]$ to the identity operator on $\mathcal{H}_\Phi$, and so the restriction of $\eusc{A}_\Phi$ to $\mathbb{Z}[1/6]$ factors through the quotient $\mathbb{Z}[1/6]/r\mathbb{Z}[1/6] \cong \mathbb{Z}/|r|\mathbb{Z}$. Thus the restriction of $\eusc{A}_\Phi$ to $C^\ast(\mathbb{Z}[1/6])$ factors through the $r$-dimensional algebra $C^\ast(\mathbb{Z}/|r|\mathbb{Z})$, and we find $\operatorname{cndm}(\Phi) \leq |r|$. This completes the proof of Proposition \ref{prop.equalsone}. \end{proof}

\subsubsection{Bounds relating conditional finite dimension and support} \label{sec.boundstst}

In Section \ref{sec.boundstst} we prove the inequalities which will allow us to relate cardinalities of supports and conditional finite dimensions.

\begin{proposition} \label{prop.conddim} Let $\Phi \in \operatorname{TSt}(C^\ast(\mathfrak{F}_{2,3}))$ be a conditionally finite-dimensional tracial state. Then $\mu_\Phi$ is finitely supported, and moreover we have:
\begin{equation} \label{eq.cndm} \lvert \mathrm{supp}(\mu_\Phi) \rvert \leq \mathrm{cndm}(\Phi) \leq 6^{\lvert \mathrm{supp}(\mu_\Phi) \rvert} \end{equation} \end{proposition}

\begin{proof}[Proof of Proposition \ref{prop.conddim}] Let $\Phi \in \operatorname{TSt}(C^\ast(\mathfrak{F}_{2,3}))$ be conditionally finite-dimensional, and let $d \in \mathbb{N}$ be the conditional dimension of $\Phi$. Also let $\mathcal{H}$ be the underlying Hilbert space of the GNS representation of $\Phi$ and let $\eusc{A}:C^\ast(\mathfrak{F}_{2,3}) \to \operatorname{End}(\mathcal{H})$ be the representation itself. We first establish the left inequality in (\ref{eq.cndm}).\\
\\
Writing $\mathsf{Z}_\mathcal{H}$ for the zero operator on $\mathcal{H}$, our hypothesis on the conditional dimension of $\Phi$ implies there exist complex numbers $c_0,\ldots,c_d$, not all zero, such that the following holds.
\begin{align*} \eusc{A} \left( \sum_{k=0}^d  c_k \upsilon(0,0;k) \right) = \mathsf{Z}_\mathcal{H} \end{align*}
Letting $\Omega:G \to \mathcal{H}$ be as in Proposition \ref{prop.posdef}, the last display implies:
\begin{align*} \nml \eusc{A} \left( \sum_{k=0}^d  c_k \upsilon(0,0;k) \right) \Omega(\iota_G) \nmr^2 = 0 \end{align*}
Unwrapping the definition of the norm on $\mathcal{H}$, we find:
\begin{align*} \sum_{k,\ell=0}^d c_k \ov{c_\ell} \Phi(0,0;k-\ell) = 0 \end{align*}
Since $\Phi(0,0;n) = \hat{\mu_\Phi}(n)$ for all $n \in \mathbb{Z}$, we can combine the last display with Proposition \ref{prop.zeros} to find that $\lvert \operatorname{supp}(\mu_\Phi)| \leq d$.\\
\\
It remains to establish the right inequality in (\ref{eq.cndm}). Recall that for a finitely supported $\mu \in \prob(\mathbb{T};2,3)$, the quantity $\operatorname{maxr}(\mu)$ was introduced in Definition \ref{def.maxr}. Writing $r$ for $\operatorname{maxr}(\mu_\Phi)$, in light of Proposition \ref{prop.maxr} it suffices to show $r \geq d$. By construction, we have $\omega^r = 1$ for all $\omega \in \operatorname{supp}(\mu)$, and so $\hat{\mu_\Phi}(r\ell) =1$ for all $\ell \in \mathbb{Z}$.  This completes the proof of Proposition \ref{prop.conddim}. \end{proof}

The proof of Proposition \ref{prop.conddim} also implies that if $\mu$ is finitely supported then $\Phi^\mu$ is conditionally finite-dimensional. In light of Remark \ref{rem.partinv}, we obtain the following.

\begin{corollary} \label{cor.condim} Let $\mu \in \prob(\mathbb{T};2,3)$ be finitely supported. Then $\Phi^\mu$ is conditionally finite dimensional, and moreover we have:
\begin{equation} \label{eq.cndm} \lvert \mathrm{supp}(\mu) \rvert \leq \mathrm{cndm}(\Phi^\mu) \leq 6^{\lvert \mathrm{supp}(\mu) \rvert} \end{equation}  \end{corollary}

\subsection{Induction of characters}

\subsubsection{Classification of characters on metabelian groups} \label{subsec.induction}

In Section \ref{subsec.induction} we outline certain facts about induction of characters, following the exposition in Section 3 of \cite{MR4753084}.  (Note that this reference uses the word 'character' to mean what we would refer to as an 'extreme character'.) This material will be necessary to deal with the fact that the correspondence between measures and tracial states is not bijective. Let $G$ be a countable discrete group, let $H$ be a subgroup of $G$ and let $\chi:H \to \mathbb{C}$ be a character on $H$. We make the following definitions.
\begin{itemize} \item We define the \emph{trivial extension of $\chi$} to be the positive definite function $\chi_{\uparrow}:G \to \mathbb{C}$ given for $g \in G$ by:
\begin{align*} \chi_\uparrow(g) \coloneqq \begin{cases} \chi(g) & \mbox{if }g \in G \\ 0 & \mbox{else} \end{cases}\end{align*}
 \item We define the \emph{isotropy group of $\chi$ in $G$} to be the subgroup $\operatorname{iso}_G(\chi)$ of $G$ defined by:
\begin{align*} \operatorname{iso}_G(\chi) \coloneqq \{g \in G: \chi(ghg^{-1}) = \chi(h) \mbox{ for all }h \in H\} \end{align*}
\item We define $\chi$ to be \emph{almost $G$-invariant} if the subgroup $\operatorname{iso}_G(\chi)$ as finite index in $G$.
\item Assuming $\chi$ is almost $G$-invariant, let $n \in \mathbb{N}$ be the index of $\operatorname{iso}_G(\chi)$ in $G$ and let $g_1,\ldots,g_n$ be a coset transversal for $G/\operatorname{iso}_G(\chi)$. We define the \emph{induction of $\chi$ to $G$} to be the character $\operatorname{Ind}_H^G(\chi):G \to \mathbb{C}$ given for $g \in G$ by:
\begin{align*} [\operatorname{Ind}_H^G(\chi)](g) \coloneqq \frac{1}{n} \sum_{m=1}^n \chi_\uparrow(g_m h g_m^{-1}) \end{align*} \end{itemize}
It is straightforward to verify that the definition of $\operatorname{Ind}_H^G(\chi)$ is independent of the choice of coset transversal for $G/\operatorname{iso}_G(\chi)$. We now have the following interesting result, which appears as Theorem 4.1 in \cite{MR4753084}.

\begin{theorem}[Levit, Vigdorovich] \label{thm.induction} Let $G$ be a metabelian group, let $H \trianglelefteq G$ be a subgroup of $G$ such that $G/H$ is abelian and let $\chi:G \to \mathbb{C}$ be an extreme normalized character. Then there exists a subgroup $M \leq G$ which contains $H$ and an almost $G$-invariant normalized character $\tau:M \to \mathbb{C}$ such that $\chi = \operatorname{Ind}_M^G(H)$ and such that $\tau(kmkm^{-1}) = 1$ for all $k,m \in M$. \end{theorem}

\subsubsection{Recoverable states} \label{sec.recover}

The existence of elements of $\Phi \in \operatorname{TSt}(C^\ast(\mathfrak{F}_{2,3}))$ for which the following definition fails captures the failure of injectivity for the map $\Phi \mapsto \mu_\Phi$.

\begin{definition} We define a character $\chi:\mathfrak{F}_{2,3} \to \mathbb{C}$ to be \emph{recoverable} if for all $j,k \in \mathbb{Z}$ and all $q \in \mathbb{Z}[1/6]$ we have $\chi(j,k;q) = 0$ unless $j=k=0$. In accordance with Remark \ref{rem.stchar}, we define a tracial state $\Phi \in \operatorname{TSt}(C^\ast(\mathfrak{F}_{2,3}))$ to be \emph{recoverable} if its associated character is recoverable. \end{definition}

\begin{remark} \label{rem.recoverable} It is immediate from the construction that a tracial state $\Phi \in \operatorname{TSt}(C^\ast(\mathfrak{F}_{2,3}))$  is recoverable if and only if we have $\Phi^{\mu_\Phi} = \Phi$. \end{remark}

In the remainder of Section \ref{sec.recover} we establish some results which will allow us to get around the existence of tracial states which are not recoverable.

\begin{proposition} \label{prop.emform} Let $M$ be a subgroup of $\mathfrak{F}_{2,3}$ which contains $\mathbb{Z}[1/6]$. Then $M$ has the form $(M \cap \mathbb{Z}^2) \ltimes \mathbb{Z}[1/6]$. Moreover, $M$ is normal in $G$. \end{proposition}

\begin{proof}[Proof of Proposition \ref{prop.emform}] Let $(j,k;q) \in M$. Since we have assumed $\mathbb{Z}[1/6] \leq M$ it follows that $(0,0;-q) \star (j,k;q) = (j,k;0) \in M$. Therefore we see $M = (M \cap \mathbb{Z}^2) \ltimes \mathbb{Z}[1/6]$, and the normality statement follows immediate from this. \end{proof}

\begin{proposition} \label{prop.abelian} Let $\chi:\mathfrak{F}_{2,3} \to \mathbb{C}$ be a character which is not recoverable, and let $j,k \in \mathbb{Z}$ be such that $jk \neq 0$ and there exists $q \in \mathbb{Z}[1/6]$ with $\chi(j,k;q) \neq 0$. Then there exists a nonzero $n \in \mathbb{N}$ with $n \leq 2^{|j|}3^{|k|}$ such that $\chi(0,0;n) = 1$. \end{proposition}

\begin{proof}[Proof of Proposition \ref{prop.abelian}] Instantiate Theorem \ref{thm.induction} with the choices $G \coloneqq \mathfrak{F}_{2,3}$ and $H \coloneqq \mathbb{Z}[1/6]$. We obtain a subgroup $M$ containing $\mathbb{Z}[1/6]$ and a character $\tau:M \to \mathbb{C}$ such that $\chi = \operatorname{Ind}_M^G(\tau)$. Since Proposition \ref{prop.emform} asserts that $M$ is normal in $\mathfrak{F}_{2,3}$, it follows from the induction procedure that we must have $(j,k;q) \in M$. Then for any $p \in \mathbb{Z}[1/6]$ we have:
\begin{align*} (j,k;q) \star (0,0;p) \star (j,k;q)^{-1} \star (0,0;p)^{-1} = (0,0;(2^j3^k-1)p) \end{align*}
Thus Theorem \ref{thm.induction} implies that $\tau(0,0;(2^j3^k-1)p) = 1$ for all $p \in \mathbb{Z}[1/6]$. Since the subgroup $(2^j3^k-1) \mathbb{Z}[1/6]$ of $\mathfrak{F}_{2,3}$ is normal, it follows that $\tau(g \star (0,0;(2^j3^k-1)p) \star g^{-1}) = 1$ for all $g \in \mathfrak{F}_{2,3}$. This we find $\chi(0,0;(2^j3^k-1)p) = 1$ for all $p \in \mathbb{Z}[1/6]$. Taking $p = -2^{|j|}3^{|k|}$ and $n = 2^{|j|}3^{|k|} - 2^{|j|+j}3^{|k|+k}$ completes the proof of Proposition \ref{prop.abelian}. \end{proof}

From Propositions \ref{prop.equalsone} and \ref{prop.abelian} we obtain the following.

\begin{corollary} \label{cor.condim-1} Suppose $\Phi \in \operatorname{TSt}(C^\ast(\mathfrak{F}_{2,3}))$ is not recoverable. Then $\Phi$ is conditionally finite-dimensional and we have $\Phi(j,k;q) = 0$ for all $j,k \in \mathbb{Z}$ such that $jk \neq 0$ and $ 2^{|j|}3^{|k|} < \operatorname{cndm}(\Phi)$, and all $q \in \mathbb{Z}[1/6]$. \end{corollary}

\subsection{Continuity and extremeness}

\subsubsection{Continuity of the tracial correspondence}

The result below shows tracial correspondence is compatible with the vague and pointwise-convergence topologies, although the failure of injectivity means it is not a homeomorphism.

\begin{proposition} \label{prop.contstate} \begin{description} \item[(Clause I)] The map $\mu \mapsto \Phi^\mu$ is continuous from the vague topology on $\prob(\mathbb{T};2,3)$ to the topology of pointwise convergence on $\operatorname{TSt}(C^\ast(\mathfrak{F}_{2,3}))$.
\item[(Clause II)] The map $\Phi \mapsto \mu_\Phi$ is continuous from the topology on pointwise convergence on $\operatorname{TSt}(C^\ast(\mathfrak{F}_{2,3}))$ to the the vague topology on $\prob(\mathbb{T};2,3)$.  \end{description} \end{proposition}

\begin{proof}[Proof of Proposition \ref{prop.contstate}] We first address Clause I. Let $(\mu_n)_{n \in \mathbb{N}}$ be a sequence of elements of $\prob(\mathbb{T};2,3)$ which converges in the vague topology to $\mu \in \prob(\mathbb{T};2,3)$. Taking into account Remark \ref{rem.vague}, we see that $\lim_{n \to \infty} \Phi^{\mu_n}(0,0;\ell) = \Phi^\mu(0,0;\ell)$ for all $\ell \in \mathbb{Z}$. The extension procedure used in Section \ref{subsec.extend} then gives that $\lim_{n \to \infty} \Phi^{\mu_n}(0,0;q) = \Phi^{\mu}(0,0;q)$ for all $q \in \mathbb{Z}[1/6]$. Since $\Phi^{\mu_n}(j,k;\ell) = 0 = \Phi^\mu(j,k;\ell)$ unless $j,k \in \mathbb{Z}$ satisfy $j=k=0$, we find that $\lim_{n \to \infty} \Phi^{\mu_n}(j,k;q) = \Phi^{\mu}(j,k;q)$ for all $j,k \in \mathbb{Z}$ and all $q \in \mathbb{Z}[1/6]$. Since this convergence holds on a subset of $C^\ast(\mathfrak{F}_{2,3})$ whose linear span is norm-dense, we find that $\lim_{n \to \infty} \Phi^{\mu_n}(s) = \Phi^{\mu}(s)$ as required. Clause II follows immediately from Remark \ref{rem.vague}, and so the proof of Proposition \ref{prop.contstate} is complete. \end{proof}

\subsubsection{Preservation of extremeness}

The existence of nonrecoverable states makes the following proposition more complicated to prove than its analog Proposition \ref{prop.convexity}. 

\begin{proposition} \label{prop.convexity-well} \begin{description} \item[(Clause I)] If $\mu \in \prob(\mathbb{T};2,3)$ is extreme then $\Phi^\mu \in \operatorname{TSt}(C^\ast(\mathfrak{F}_{2,3}))$ is extreme. 
\item[(Clause II)] If $\Phi \in \operatorname{TSt}(C^\ast(\mathfrak{F}_{2,3}))$ is extreme then $\mu_\Phi \in \prob(\mathbb{T};2,3)$ is extreme. \end{description} \end{proposition}

\begin{proof}[Proof of Proposition \ref{prop.convexity-well}] It follows immediately from the construction that if $\mu,\nu \in \prob(\mathbb{T};2,3)$ and $t \in [0,1]$ then we have $\Phi^{t\mu + (1-t)\nu} = t \Phi^\mu + (1-t)\Phi^\nu$. Taking into account Remark \ref{rem.partinv}, we obtain the statement of Clause I. Since it is not necessarily the case that $\Phi^{\mu_\Phi} = \Phi$, we must give a more elaborate argument for Clause II.\\
\\
Suppose toward a contradiction that $\Phi \in \operatorname{TSt}(C^\ast(\mathfrak{F}_{2,3}))$ is extreme but we have a decomposition $\mu_\Phi = t\kappa + (1-t)\nu$ for some $\kappa,\nu \in \prob(\mathbb{T};2,3)$ and $t \in (0,1)$. Also let $\chi:\mathfrak{F}_{2,3} \to \mathbb{C}$ denote the restriction of $\Phi$, so that $\chi$ is an extreme character on the group $\mathfrak{F}_{2,3}$. Apply Theorem \ref{thm.3} to obtain $M$ and $\tau$. The last clause in Theorem \ref{thm.1} implies that we can regard $\tau$ as the Fourier transform of a probability measure $\vartheta$ on the dual group of the abelianization of $M$. Moreover, Proposition \ref{prop.emform} implies that this dual group has the form $\mathbb{T}^b \times S$ for some $b \in \{0,1,2\}$, where $S$ is the dual group of $\mathbb{Z}[1/6]$. Moreover, since $\tau$ was assumed to be a character, the measure $\vartheta$ is invariant under the action of $\mathfrak{F}_{2,3}$ on $\mathbb{T}^b \times S$ which corresponds to the product of the trivial action on $\mathbb{T}^b$ and the $\times (2,3)$ action on $S$. \\
\\
Now, identify $\mathbb{T}$ as a quotient of $S$ via the identification of $\mathbb{Z}$ with a subgroup of $\mathbb{Z}[1/6]$, and thereby obtain an identification of $\mathbb{T}^b \times \mathbb{T}$ with a quotient of $\mathbb{T}^b \times S$. By hypothesis the pushforward of $\vartheta$ onto the rightmost factor in the product $\mathbb{T}^b \times \mathbb{T}$ is equal to $\mu_\Phi$. This we can consider the disintegration of $\vartheta$ over $\mu_\Phi$, which provides a probability measure $\vartheta_\omega$ on $\mathbb{T}^b \times S$ for every $\omega \in \mathbb{T}$ such that the formula below holds.
\begin{align*} \vartheta = \int_{\mathbb{T}} \vartheta_\omega \deee \mu_\Phi(\omega)  \end{align*}
Our hypothesis on the convex decomposition of $\mu_\Phi$ allows us to express the above integral as:
\begin{align*} \vartheta = t \int_{\mathbb{T}} \vartheta_\omega \deee \kappa(\omega)  + (1-t) \int_{\mathbb{T}} \vartheta_\omega \deee \nu (\omega)  \end{align*}
Since the action of $\mathfrak{F}_{2,3}$ on $\mathbb{T}^b$ is trivial and $\kappa$ and $\nu$ are $\times (2,3)$-invariant, the two integrated measures in the convex decomposition on the right of the last display are invariant under the action of $\mathfrak{F}_{2,3}$ on $\mathbb{T}^b \times S$. Taking inverse Fourier transforms, we obtain a nontrivial convex decomposition of $\tau$ into two characters of $N$. This decomposition can be induced to obtain a nontrivial convex decomposition of $\Phi$ into two elements of $\operatorname{TSt}(C^\ast(\mathfrak{F}_{2,3})$, contradicting the extremeness of $\Phi$. This completes the proof of Proposition \ref{prop.convexity-well}. \end{proof}

\subsection{Main argument for equivalence of ergodic and $C^\ast$-algebraic conjectures}

\paragraph{Deduction of Conjecture \ref{conj.erg-1} from Conjecture \ref{conj.see-1}} Assume Conjecture \ref{conj.see-1} and let $\mu \in \prob(\mathbb{T}; 2,3)$ be extreme. Then Clause I in Proposition \ref{prop.convexity-well} implies that $\Phi^\mu$ is a extreme element of $\operatorname{TSt}(C^\ast(\mathfrak{F}_{2,3}))$, and so Conjecture \ref{conj.see-1} implies that $\Phi^\mu = \Delta$ or else $\Phi^\mu$ is conditionally finite-dimensional. If $\Phi^\mu = \Delta$ then it is immediate that $\mu = \lambda$. On the other hand, if $\Phi^\mu$ is conditionally finite-dimensional then Proposition \ref{prop.conddim} implies that $\mu_{\Phi^\mu}$ is finitely supported. According to Remark \ref{rem.partinv} we have $\mu_{\Phi^\mu} = \mu$, and so we find that Conjecture \ref{conj.see-1} implies Conjecture \ref{conj.erg-1}, as required.

\paragraph{Deduction of Conjecture \ref{conj.see-1} from Conjecture \ref{conj.erg-1}} Assume Conjecture \ref{conj.erg-1}, and let $\Phi \in \operatorname{TSt}(C^\ast(\mathfrak{F}_{2,3}))$ be extreme. We may assume without loss of generality that $\Phi$ is conditionally infinite-dimensional, in order to show that $\Phi = \Delta$. Combining this assumption with Corollary \ref{cor.condim}, we see that $\Phi$ is recoverable. Then Clause II in Proposition \ref{prop.convexity} implies that $\mu_\Phi$ is an ergodic element of $\prob(\mathbb{T};2,3)$, and so Conjecture \ref{conj.erg-1} implies that $\mu_\Phi = \lambda$ or $\mu_\Phi$ is finitely supported. \\
\\
The assumption that $\Phi$ is recoverable implies that $\Phi^{\mu_\Phi} = \Phi$ in accordance with Remark \ref{rem.recoverable}. Since $\Phi$ was assumed to be conditionally infinite-dimensional, Corollary \ref{cor.condim} implies $\mu_\Phi$ cannot be finitely supported. Thus Conjecture \ref{conj.erg-1} implies that $\mu_\Phi = \lambda$ and we find $\Phi = \Phi^\lambda = \Delta$. Thus Conjecture \ref{conj.erg-1} indeed implies Conjecture \ref{conj.see-1}.

\paragraph{Deduction of Conjecture \ref{conj.erg-2} from Conjecture \ref{conj.see-2}} Assume Conjecture \ref{conj.see-2}, and let $\mu \in \prob(\mathbb{T};2,3)$ be arbitrary. By hypothesis we can find a sequence $(\Phi_n)_{n \in \mathbb{N}}$ of conditionally finite dimensional elements of $\prob(\mathbb{T};2,3)$ which converges pointwise to $\Phi^\mu$. According to Proposition \ref{prop.conddim} we have that $\mu_{\Phi_n}$ is finitely supported for all $n \in \mathbb{N}$, and Clause II in Proposition \ref{prop.contstate} implies that $(\mu_{\Phi_n})_{n \in \mathbb{N}}$ converges vaguely to $\mu_{\Phi^\mu}$. According to Remark \ref{rem.partinv} we have $\mu_{\Phi^\mu} = \mu$, and we find that Conjecture \ref{conj.see-2} indeed implies Conjecture \ref{conj.erg-2}.

\paragraph{Deduction of Conjecture \ref{conj.see-2} from Conjecture \ref{conj.erg-2}} Assume Conjecture \ref{conj.erg-2}, and let $\Phi \in \operatorname{TSt}(C^\ast(\mathfrak{F}_{2,3}))$. We may assume without loss of generality that $\Phi$ is extreme and conditionally infinite-dimensional. Thus Corollary \ref{cor.condim} implies that $\Phi$ is recoverable. By hypothesis we can find a sequence of finitely supported measures $(\mu_n)_{n \in \mathbb{N}}$ converging vaguely to $\mu_\Phi$. Then Corollary \ref{cor.condim} implies that $\Phi^{\mu_n}$ is conditionally finite-dimensional for all $n \in \mathbb{N}$, and Clause I in Proposition \ref{prop.contstate} implies that $\Phi^{\mu_n}$ converges pointwise to $\Phi^{\mu_\Phi}$. Since $\Phi$ was assumed to be recoverable we find that $\Phi^{\mu_n}$ converges pointwise to $\Phi$, and thus we see that Conjecture \ref{conj.erg-2} indeed implies Conjecture \ref{conj.see-2}.

\paragraph{Deduction of Conjecture \ref{conj.erg-3} from Conjecture \ref{conj.see-3}} Assume Conjecture \ref{conj.see-3}, and let $(\mu_n)_{n \in \mathbb{N}}$ be a sequence of extreme elements of $\prob(\mathbb{T})$ with $\lim_{n \to \infty} \lvert \operatorname{supp}(\mu_n) \rvert = \infty$. Then Corollary \ref{cor.condim} implies that $(\Phi^{\mu_n})_{n \in \mathbb{N}}$ is a sequence of extreme conditionally finite-dimensional elements of $\operatorname{TSt}(C^\ast(\mathfrak{F}_{2,3}))$ such that $\lim_{n \to \infty} \operatorname{cndm}(\Phi^{\mu_n}) = \infty$. Thus Conjecture \ref{conj.see-3} implies that $\Phi^{\mu_n}$ converges pointwise to $\Delta$, and so Clause II in Proposition \ref{prop.contstate} implies that $\mu_{\Phi^{\mu_n}}$ converges vaguely to $\mu_\Delta = \lambda$. According to Remark \ref{rem.partinv} we have $\mu_{\Phi^{\mu_n}} = \mu_n$ for all $n \in \mathbb{N}$, and so we find that Conjecture \ref{conj.see-3} indeed implies Conjecture \ref{conj.erg-3}. 

\paragraph{Deduction of Conjecture \ref{conj.see-3} from Conjecture \ref{conj.erg-3}} Assume Conjecture \ref{conj.erg-3} and let $(\Phi_n)_{n \in \mathbb{N}}$ be a sequence of extreme conditionally finite-dimensional elements of $\operatorname{TSt}(C^\ast(\mathfrak{F}_{2,3}))$ such that $\lim_{n \to \infty} \operatorname{cndm}(\Phi_n) = \infty$. By splitting this sequence into two subsequences, we may assume without loss of generality that either $\Phi_n$ is recoverable for all $n \in \mathbb{N}$ or $\Phi_n$ fails to be recoverable for all $n \in \mathbb{N}$. We now address each of these cases separately.

\begin{itemize} \item Assume that $\Phi_n$ is recoverable for all $n \in \mathbb{N}$. Then Proposition \ref{prop.conddim} implies that $(\mu_{\Phi_n})_{n \in \mathbb{N}}$ is a sequence of ergodic finitely supported elements of $\prob(\mathbb{T};2,3)$ such that $\lim_{n \to \infty} \lvert \operatorname{supp}(\mu_{\Phi_n})\rvert = \infty$. Thus Conjecture \ref{conj.erg-3} implies that $(\mu_{\Phi_n})_{n \in \mathbb{N}}$ converges vaguely to $\lambda$, and Clause I in Proposition \ref{prop.contstate} implies that $\Phi^{\mu_{\Phi_n}}$ converges pointwise to $\Delta$. Since we assumed that $\Phi_n$ was recoverable, we have $\Phi^{\mu_{\Phi_n}} = \Phi_n$ for all $n \in \mathbb{N}$, and this we find that $(\Phi_n)_{n \in \mathbb{N}}$ converges pointwise to $\Delta$ in this case.

\item Assume that $\Phi_n$ fails to be recoverable for all $n \in \mathbb{N}$. Then Corollary \ref{cor.condim-1} implies that $\lim_{n \to \infty} \Phi_n(j,k;q) = 0$ for all $j,k \in \mathbb{Z}$ such that $jk \neq 0$ and all $q \in \mathbb{Z}[1/6]$. Since the sequence $(\Phi_n)_{n \in \mathbb{N}}$ converges pointwise to $\Delta$ on a subset of $C^\ast(\mathfrak{F}_{2,3})$ whose linear span is norm-dense, we find that $(\Phi_n)_{n \in \mathbb{N}}$ pointwise converges to $\Delta$ in this case as well. \end{itemize}

Thus we see that Conjecture \ref{conj.erg-3} indeed implies Conjecture \ref{conj.see-3}.

\bibliographystyle{plain}

\bibliography{bibliography-2024_10_29.bib}

\hrulefill

\vspace{0.5 cm}

\textbf{Peter Burton:} \\
Department of Mathematics and Statistics\\
University of Wyoming, Laramie WY \\
\url{peterburton1728@gmail.com}\\
\\
\textbf{Jane Panangaden:} \\
Mathematics Field Group\\
Pitzer College, Claremont CA \\
\url{jane.panangaden@gmail.com}

\end{document}